\documentclass[twoside,a4paper,reqno,11pt]{amsart} 
\usepackage[top=28mm,right=28mm,bottom=28mm,left=28mm]{geometry}

\usepackage{amsfonts, amsmath, amssymb, mathrsfs, bm, latexsym, stmaryrd, array, hyperref, mathtools, bold-extra,color}

\renewcommand{\a}{\alpha}
\renewcommand{\b}{\beta}

\newcommand{\e}{\varepsilon}

\renewcommand{\O}{\Omega}

\newcommand{\normeq}{\trianglelefteqslant}

\newcommand{\la}{\langle}
\newcommand{\ra}{\rangle}

\newcommand{\leqs}{\leqslant}
\newcommand{\geqs}{\geqslant}

\newcommand{\vs}{\vspace{2mm}}

\makeatletter
\newcommand{\imod}[1]{\allowbreak\mkern4mu({\operator@font mod}\,\,#1)}
\makeatother

\newtheorem{theorem}{Theorem} 
\newtheorem{theoremm}{Theorem}

\newtheorem{corr}[theoremm]{Corollary}

\newtheorem*{conj*}{Conjecture}

\newtheorem{conj}[theorem]{Conjecture}

\newtheorem{thm}{Theorem}[section] 
\newtheorem{prop}[thm]{Proposition} 
\newtheorem{lem}[thm]{Lemma}

\theoremstyle{definition}
\newtheorem{rem}[thm]{Remark}
\newtheorem{remk}[theorem]{Remark}

\begin{document}

\author{Timothy C. Burness}
\address{T.C. Burness, School of Mathematics, University of Bristol, Bristol BS8 1UG, UK}
\email{t.burness@bristol.ac.uk}

\author{Hong Yi Huang}
\address{H.Y. Huang, School of Mathematics and Statistics, University of St Andrews, KY16 9SS, UK}
\curraddr{Department of Mathematics, Southern University of Science and Technology, Shenzhen 518055, Guangdong, P. R. China}
\email{11612012@mail.sustech.edu.cn}
 
\title[On the intersections of Sylow subgroups in almost simple groups]{On the intersections of Sylow subgroups \\ in almost simple groups}

\dedicatory{\rm Dedicated to Professor Viktor Zenkov} 

\begin{abstract}
Let $G$ be a finite almost simple group and let $H$ be a Sylow $p$-subgroup of $G$.
As a special case of a theorem of Zenkov, there exist $x,y \in G$ such that $H \cap H^x \cap H^y = 1$. In fact, if $G$ is simple, then a theorem of Mazurov and Zenkov reveals that $H \cap H^x = 1$ for some $x \in G$. However, it is known that the latter property  does not extend to all almost simple groups. For example, if $G = S_8$ and $p=2$, then $H \cap H^x \ne 1$ for all $x \in G$. Further work of Zenkov in the 1990s shows that such examples are rare (for instance, there are no such examples if $p \geqslant 5$) and he reduced the classification of all such pairs to the situation where $p=2$ and $G$ is an almost simple group of Lie type defined over a finite field $\mathbb{F}_q$ and either $q=9$ or $q$ is a Mersenne or Fermat prime. In this paper, by adopting a probabilistic approach based on fixed point ratio estimates, we complete Zenkov's classification.
\end{abstract}

\date{\today}

\maketitle

\section{Introduction}\label{s:intro}

This paper is a contribution to an extensive and expanding literature on problems concerning the intersections of nilpotent subgroups in finite groups, which stretches  back several decades. In general terms, our main problem of interest can be stated as follows: given a finite group $G$ and a collection of nilpotent subgroups $H_1, \ldots, H_k$, can we find  elements $x_i \in G$ such that the intersection of the conjugate subgroups $\bigcap_i H_i^{x_i}$ is as small as possible, in some natural sense? For example, a theorem of Passman \cite{Pass} from the 1960s states that if $G$ is $p$-soluble and $H$ is a Sylow $p$-subgroup of $G$ for some prime $p$, then there exist elements $x,y \in G$ such that 
\[
H \cap H^x \cap H^y = O_p(G),
\]
where $O_p(G)$ denotes the largest normal $p$-subgroup of $G$ (and is therefore contained in every conjugate of $H$). 

In more recent years, Passman's theorem has been extended and generalised in various directions. For example, Zenkov uses the Classification of Finite Simple Groups in \cite{Z_92p} to extend Passman's theorem to all finite groups, without any assumption on the $p$-solubility of $G$. And in \cite{Dolfi}, Dolfi proves that if $G$ is $\pi$-soluble and $H$ is a Hall $\pi$-subgroup, with $\pi$ any set of primes, then $H \cap H^x \cap H^y = O_{\pi}(G)$ for some $x,y \in G$ (here Passman's theorem is the special case $\pi = \{p\}$). In a different direction, Zenkov \cite{Z_3} shows that if $G$ is a finite group with nilpotent subgroups $A,B,C$, then there exist $x,y \in G$ such that 
\[
A \cap B^x \cap C^y \leqs F(G),
\]
where $F(G)$ is the Fitting subgroup of $G$. It is not too difficult to see that the latter statement is false if we only consider two nilpotent subgroups. For example, if $G = S_8$ is the symmetric group of degree $8$ and $A = B$ is a Sylow $2$-subgroup, then $F(G) = 1$ and $A \cap B^x \ne 1$ for all $x \in G$. We refer the reader to \cite{Z_ab2, Z_92p, Z_ab} for some further variations on this theme.

In this paper, we will focus on the intersections of Sylow subgroups. So let us fix a finite group $G$ and let $H$ be a Sylow $p$-subgroup of $G$, where $p$ is a prime divisor of $|G|$. Let us assume that $O_p(G) = 1$, which means that $H$ is a \emph{core-free}  subgroup of $G$. In particular, it follows that $G$ acts faithfully on the set $\O = G/H$ of cosets of $H$ in $G$ and this allows us to view $G \leqs {\rm Sym}(\O)$ as a transitive permutation group on $\O$ with point stabiliser $H$. In this setting, we recall that a \emph{base} for $G$ is a subset $B$ of $\O$ with the property that the pointwise stabiliser of $B$ in $G$ is trivial. And then the \emph{base size} of $G$, denoted $b(G,H)$, is defined to be the minimal size of a base and we note that 
\[
b(G,H) = \min\left\{ |S| \,:\, S \subseteq G, \, \bigcap_{x \in S} H^x = 1\right\}.
\]
Determining the base size of a given group is a fundamental problem in permutation group theory and there is a vast literature on this topic, dating back more than a century, with bases finding natural connections and applications across many different areas of group theory and combinatorics. We refer the reader to the survey articles \cite{BC,Maroti}  and \cite[Section 5]{B180} for further details and references on bases and their applications.

In the language of bases, Zenkov's aforementioned theorem from \cite{Z_3} implies that if $G$ is a finite transitive permutation group with $F(G) = 1$ and a nilpotent point stabiliser $H$, then $b(G,H) \leqs 3$ and this bound is best possible (for instance, if $G = S_8$ and $H$ is a Sylow $2$-subgroup, then $b(G,H) = 3$). The analogous problem for transitive groups with soluble point stabilisers is less well understood. Here the main conjecture is due to Vdovin, which asserts that $b(G,H) \leqs 5$ if $R(G) = 1$, where $R(G)$ is the soluble radical of $G$ (see Problem 17.41(b) in the \emph{Kourovka Notebook} \cite{Kou}). In other words, if $G$ is a finite group with $R(G) = 1$ and $H$ is a soluble subgroup, then it is conjectured that there exist elements $x_1, \ldots, x_4 \in G$ such that 
\[
H \cap H^{x_1} \cap H^{x_2} \cap H^{x_3} \cap H^{x_4} = 1.
\]
Once again, examples show that this proposed upper bound is optimal (for instance, if $G = S_8$ and $H = S_4 \wr S_2$, then $b(G,H) = 5$). This conjecture remains open, although there is a reduction to almost simple groups due to Vdovin \cite{Vdovin}, and there is a further reduction to groups of Lie type \cite{Bay1,B23}, with recent progress on classical groups due to Baykalov \cite{Bay2}. The conjecture in the special case where $H$ is a soluble maximal subgroup of $G$ is proved in \cite{Bur21} and we refer the reader to \cite[Section 1]{AB} for various extensions of a similar flavour.

Let us now focus our attention on the case where $G$ is a finite almost simple group with socle $G_0$. Here $G_0$, which is the unique minimal normal subgroup of $G$, is a non-abelian simple group and by identifying $G_0$ with its group of inner automorphisms we have
\[
G_0 \normeq G \leqs {\rm Aut}(G_0).
\]
Let $H$ be a non-trivial nilpotent subgroup of $G$. Then $b(G,H) \leqs 3$ by Zenkov's theorem \cite{Z_3} and we highlight the following conjecture in the special case where  $G$ is simple (see Problem 15.40 in \cite{Kou}):

\begin{conj}[Vdovin \cite{Kou}, 2002]\label{c:nilp}
If $G$ is a finite non-abelian simple group and $H$ is a non-trivial nilpotent subgroup, then $b(G,H) = 2$. 
\end{conj}

This conjecture has been resolved in various special cases (see \cite{Kurm,Z_spor} for alternating and sporadic groups, for example), but it remains open for simple groups of Lie type. We refer the reader to Remarks \ref{r:ls} and \ref{r:ab} for a brief discussion of some related open problems.

The case where $H$ is a Sylow $p$-subgroup of $G$ for some prime divisor $p$ of $|G|$ is of course a natural special case of Conjecture \ref{c:nilp}. In this setting, the conjecture was proved by Mazurov and Zenkov \cite{ZM} in 1996. It is interesting to note that their proof for groups of Lie type relies on several deep results in representation theory. More precisely, a theorem of Green \cite{Green} from 1962 reveals that if $G$ is a finite group, $H$ is a Sylow $p$-subgroup and $D$ is the defect group of a $p$-block for $G$, then $H \cap H^x = D$ for some $x \in G$. This key observation is then  combined with later work of Michler \cite{Michler} and Willems \cite{Willems} in the 1980s, which shows that every finite simple group of Lie type has a $p$-block of defect zero for every prime divisor $p$ of $|G|$, whence $D=1$ is trivial and Green's theorem implies that $b(G,H) = 2$.

As noted above, there exist almost simple groups $G$ with a Sylow $p$-subgroup $H$ such that $b(G,H) = 3$. So it is natural to seek a complete classification of the pairs $(G,p)$ with this property. This challenge was taken up by Zenkov in \cite{Z_92p}, which was published in 1996. In this substantial 92-page paper, Zenkov proves a number of remarkable results, culminating in the following major reduction of the original problem.

\begin{theorem}[Zenkov \cite{Z_92p}, 1996]\label{t:zenkov}
Let $G$ be an almost simple group with socle $G_0$ and let $H$ be a Sylow $p$-subgroup of $G$, where $p$ is a prime divisor of $|G|$. Then $b(G,H) \leqs 3$, with equality possible only if $p \in \{2,3\}$ and one of the following holds:

\begin{itemize}\addtolength{\itemsep}{0.2\baselineskip}
\item[{\rm (i)}] $(p,G)$ is recorded in Table \ref{tab:zen}, in which case $b(G,H) = 3$; or

\item[{\rm (ii)}] $p=2$ and $G_0$ is a simple group of Lie type over $\mathbb{F}_q$, where $q = 9$ or $q$ is a Mersenne or Fermat prime.
\end{itemize}
\end{theorem}

\begin{table}
\[
\begin{array}{lll} \hline
p & G & \mbox{Conditions} \\ \hline
3 & {\rm P\O}_{8}^{+}(3).X & X \in \{C_3,S_3\} \\
2 & {\rm L}_2(9).2^2 & \\
& {\rm L}_2(q).2 & \mbox{$q \geqs 7$ is a Mersenne prime} \\
& {\rm L}_3(4).X & \mbox{$X \in \{C_2,C_2^2\}$, $\gamma \in G$} \\
& {\rm L}_n(2).2 & n \geqs 4 \\
& \O_{n}^{+}(2).2 & n \geqs 8 \\
& F_4(2).2 & \\
& E_6(2).2 & \\ \hline
\end{array}
\]
\caption{The almost simple groups $G$ with $b(G,H) = 3$ for $H \in {\rm Syl}_p(G)$}
\label{tab:zen}
\end{table}

\begin{remk}
Note that in Table \ref{tab:zen} we have not listed the groups ${\rm L}_3(2).2$ and $S_8$ with $p=2$. This is simply to avoid a repetition of cases, noting that ${\rm L}_3(2).2 \cong {\rm L}_2(7).2$ and $S_8 \cong {\rm L}_4(2).2$. And in the row for $G = {\rm L}_3(4).X$ we write `$\gamma \in G$' to record the fact that $G$ must contain an involutory graph automorphism of ${\rm L}_3(4)$.
\end{remk}

In fact, Theorem \ref{t:zenkov} is a special case of a more general result (see \cite[Theorem B]{Z_92p}) concerning the structure of an arbitrary finite group $G$ with a Sylow $p$-subgroup $H$ such that $H \cap H^x \ne O_p(G)$ for all $x \in G$. We refer the reader to Remark \ref{r:zenkov} for a brief overview of the main steps in the proof of Zenkov's result in the almost simple setting.

Some of the open cases arising in part (ii) of Theorem \ref{t:zenkov} have been handled in more recent work. More specifically, the groups with socle ${\rm L}_2(q)$ are treated in \cite[Lemma 3.18]{Z_92p} (in fact, the main theorem of \cite{Z_L2} determines all the pairs of nilpotent subgroups $A,B$ with $A \cap B^x \ne 1$ for all $x \in G$), and those with socle ${\rm L}_3(q)$ or ${\rm U}_3(q)$ are dealt with in \cite{Z_dim3}. In addition, the main theorem of \cite{Z_ex3} handles all the almost simple exceptional groups of Lie type over $\mathbb{F}_3$.

Our main result is Theorem \ref{t:main1} below. This shows that $b(G,H) = 2$ for all of the open cases in part (ii) of Theorem \ref{t:zenkov}, thereby completing the classification of the almost simple groups $G$ with a Sylow $p$-subgroup $H$ such that $H \cap H^x \ne 1$ for all $x \in G$. In addition, Theorem \ref{t:main1} resolves Problem 20.121 in the \emph{Kourovka Notebook} \cite{Kou}.

\begin{theoremm}\label{t:main1}
Let $G$ be an almost simple group and let $H$ be a Sylow $p$-subgroup of $G$, where $p$ is a prime divisor of $|G|$. Then $b(G,H) \leqs 3$, with equality if and only if 
$(p,G)$ is one of the cases recorded in Table \ref{tab:zen}.
\end{theoremm}

In view of Theorem \ref{t:zenkov}, we may assume $G$ is an almost simple group of Lie type over $\mathbb{F}_q$ with socle $G_0$, and either $q=9$ or $q$ is a Mersenne or Fermat prime. In addition, we may assume $H$ is a Sylow $2$-subgroup and $G_0 \ne {\rm L}_2(q)$, ${\rm L}_{3}(q)$ or ${\rm U}_3(q)$. And in view of the main theorem in \cite{ZM} for simple groups, we can also assume that $|G:G_0|$ is even. As discussed above, we may view $G \leqs {\rm Sym}(\O)$ as a transitive permutation group on $\O = G/H$ and our goal is to show that $b(G,H) = 2$. 

To prove Theorem \ref{t:main1}, we adopt a completely different approach to the problem, utilising a powerful probabilistic technique first introduced by Liebeck and Shalev \cite{LSh99}. This method has been extensively applied in recent years for studying bases for almost simple primitive permutation groups. 

This approach relies on the elementary observation that if $Q(G,H)$ is the probability that a uniformly random pair of points in $\O$ do not form a base for $G$, then
\[
Q(G,H) \leqs \sum_{i=1}^k |x_i^G| \cdot {\rm fpr}(x_i,G/H)^2 =: \widehat{Q}(G,H),
\]
where $\{x_1, \ldots, x_k\}$ is a complete set of representatives of the conjugacy classes in $G$ of elements of prime order, and 
\[
{\rm fpr}(x_i,G/H) = \frac{|x_i^G \cap H|}{|x_i^G|}
\] 
is the fixed point ratio of $x_i$ on $G/H$, which is simply the proportion of cosets in $G/H$ fixed by $x_i$. In particular, if $\widehat{Q}(G,H)<1$ then $b(G,H) = 2$. And of course, in the setting we are interested in, with $H$ a Sylow $2$-subgroup of $G$, we only need to consider fixed point ratios for involutions. 

The main challenge in applying this method involves bounding the number of involutions in $H$ that are contained in `small' conjugacy classes in $G$. To do this, it is often helpful to identify an appropriate subgroup $L<G$ containing $H$, where it may be easier to bound the number of involutions in $L$ of a given type, rather than in $H$ directly. For the remaining classes $x_i^G$ of involutions, we can often work with the trivial upper bound $|x_i^G \cap H| < |H|$, which greatly simplifies the analysis.

\vs

By combining Theorem \ref{t:main1} with more recent work of Zenkov \cite{Z_ln2, Z_omega,Z_E6,Z_L2,Z_F4,Z_odd} we immediately obtain Corollary \ref{c:main1} below, which gives a precise description of the triples $(G,A,B)$, where $G$ is almost simple and $A,B$ are primary subgroups (that is to say, $A$ and $B$ have prime power order) with the property that $A \cap B^x \ne 1$ for all $x \in G$. Note that the subgroup $K$ defined in part (ii) is denoted by ${\rm min}_G(H,H)$ in Zenkov's papers.

\begin{corr}\label{c:main1}
Let $G$ be an almost simple group and let $A$ and $B$ be primary subgroups of $G$. Then $A \cap B^x \ne 1$ for all $x \in G$ if and only if $(p,G)$ is one of the cases in Table \ref{tab:zen}, and $A,B$ are $p$-subgroups such that  
\begin{itemize}\addtolength{\itemsep}{0.2\baselineskip}
\item[{\rm (i)}] $A$ and $B$ are both Sylow $p$-subgroups of $G$; or

\item[{\rm (ii)}] $(A,B)$ is recorded in Table \ref{tab:primary}, up to conjugacy and ordering, where $H$ is a Sylow $p$-subgroup of $G$ and 
\[
K = \la H \cap H^x \,:\, x\in G, \, |H \cap H^x| = m\ra
\]
with $m = \min\{|H \cap H^x| \,:\, x \in G\}$.
\end{itemize} 
\end{corr}

\begin{table}
\[
\begin{array}{lllll} \hline
p & G & (A,B) & \mbox{Conditions} & \mbox{Reference}  \\ \hline
3 & {\rm P\O}_{8}^{+}(3).X & K \leqs A,B \leqs H & X \in \{C_3,S_3\} & \mbox{\cite[Theorem 1]{Z_odd}} \\
2 & {\rm L}_2(9).2^2 & (H,K) & & \mbox{\cite[Theorem 2]{Z_L2}} \\
& {\rm L}_3(4).X & \mbox{See Remark \ref{r:primary}(a)} & \mbox{$X \in \{C_2,C_2^2\}$, $\gamma \in G$} & \mbox{\cite[Theorem 1]{Z_ln2}} \\
& {\rm L}_n(2).2 & K \leqs A,B \leqs H & \mbox{$n \geqs 4$ even} & \mbox{\cite[Theorem 1]{Z_ln2}} \\
& \O_{n}^{+}(2).2 & K \leqs A,B \leqs H & n \geqs 8 & \mbox{\cite[Theorem 1]{Z_omega}} \\
& F_4(2).2 & (H,K) & & \mbox{\cite[Theorem 2]{Z_F4}} \\
& E_6(2).2 & K \leqs A,B \leqs H & & \mbox{\cite[Theorem 5]{Z_E6}} \\ \hline
\end{array}
\]
\caption{The pairs $(A,B)$ in part (ii) of Corollary \ref{c:main1}}
\label{tab:primary}
\end{table}

\begin{remk}\label{r:primary}
Some comments on the statement of Corollary \ref{c:main1} are in order.
\begin{itemize}\addtolength{\itemsep}{0.2\baselineskip}
\item[{\rm (a)}] Suppose $p=2$ and $G_0 = {\rm L}_3(4)$ as in Table \ref{tab:primary}. Here $G = G_0.2$ or $G_0.2^2$ contains an involutory graph automorphism $\gamma$ and we set $G_1 = G_0.\la \gamma \ra$. Then according to \cite[Theorem 1]{Z_ln2}, if $A$ and $B$ are a pair of primary subgroups of $G$, then  $A \cap B^x \ne 1$ for all $x \in G$ if and only if $K < A \cap B \cap G_1$ and either $L \leqs A$ or $L \leqs B$, where $L = \la A \cap B^x \,: \, x \in \mathcal{M}\ra$ and $\mathcal{M}$ is the set of elements $x \in G$ such that $A \cap B^x$ is inclusion-minimal in the set $\{ A \cap B^g \,:\, g \in G\}$. It is straightforward to construct all of the relevant pairs $(A,B)$ using {\sc Magma} \cite{magma}.

\item[{\rm (b)}] In each of the remaining cases, the description of the relevant pairs $(A,B)$ is easier to state and it just depends on the subgroup $K$ of $H$ defined in part (ii) of the corollary. This subgroup $K$ is given in the relevant reference listed in the final column of Table \ref{tab:primary}, but for the reader's convenience we provide a brief description here:

\vspace{2mm}

\begin{itemize}\addtolength{\itemsep}{0.2\baselineskip}
\item $G = {\rm P\O}_8^{+}(3).X$: $K = O_3(N_G(P))$, where $P = P_{1,3,4}$ is a parabolic subgroup of $G_0$ corresponding to the central node in the Dynkin diagram of type $D_4$. We note that $|H:K| = 3$.

\item $G = {\rm L}_2(9).2^2$: $K = D_{16}$ and $|H:K| = 2$.

\item $G = {\rm L}_n(2).2$, $n \geqs 4$ even: $K = O_2(N_G(P))$, where $P$ is a parabolic subgroup of $G_0$ corresponding to the central node of the Dynkin diagram of type $A_{n-1}$. Here $P = [2^a].{\rm L}_2(2)$ with $a = n(n-1)/2-1$ and thus $|H:K| = 2$.

\item $G = \O_n^{+}(2).2 = {\rm O}_n^{+}(2)$, $n \geqs 8$: $K = O_2(P)$ is the unipotent radical of a maximal parabolic subgroup $P$ of $G$. In the notation of \cite{KL}, we have $P = P_{k}$ with $k = n/2-1$, which we can view as the stabiliser in $G$ of a $k$-dimensional totally singular subspace of the natural module for $G$. In particular, $|H:K| = 2^a$ with $a = (n^2-6n+8)/8$.

\item $G = F_4(2).2$: $K = O_2(P).D_{16}$, where $P = P_{1,4}$ is a maximal parabolic subgroup of $G$ with Levi subgroup of type $B_2$. Note that $|H:K| = 2$.

\item $G = E_6(2).2$: $K = O_2(P).(O_2(Q).2)$, where $O_2(P)$ is the unipotent radical of a maximal parabolic subgroup $P =[2^{24}].{\rm O}_8^{+}(2)$ of $G$ of type $P_{1,6}$, and $O_2(Q)$ is the unipotent radical of a maximal parabolic subgroup $Q = 2^{3+6}.{\rm L}_3(2)$  of $\O_8^{+}(2)$ of type $P_1$. In particular, $|O_2(P)| = 2^{24}$ and $|O_2(Q)| = 2^9$, whence $|H:K| = 2^3$.
\end{itemize}

\vspace{1mm}

\item[{\rm (c)}] Note that if $G = {\rm L}_n(2).2$ with $n \geqs 5$ odd, then \cite[Theorem 1]{Z_ln2} implies that $A \cap B^x \ne 1$ for all $x \in G$ if and only if $A$ and $B$ are both Sylow $2$-subgroups of $G$. This explains why $n$ is even for $G = {\rm L}_n(2).2$ in Table \ref{tab:primary}. And by \cite[Theorem 2]{Z_L2}, the same conclusion holds when 
$G = {\rm L}_2(q).2$ and $q \geqs 7$ is a Mersenne prime.
\end{itemize}
\end{remk}

In order to conclude the introduction, we present several remarks describing some related open problems.

\begin{remk}\label{r:ls}
Let $G$ be a finite group and let $P_i$ be a Sylow $p_i$-subgroup of $G$, where $p_1, \ldots, p_n$ are the distinct prime divisors of $|G|$. In a recent preprint \cite{LS}, Lisi and Sabatini conjecture that there exists an element $x \in G$ such that for each $i$, $P_i \cap P_i^x$ is inclusion-minimal in the set $\{P_i \cap P_i^g \,:\, g \in G\}$. So for a simple group $G$, in view of the main theorem of \cite{ZM}, this conjecture asserts that there is an element $x \in G$ such that $P_i \cap P_i^x = 1$ for all $i$. In particular, if $H$ is a nilpotent subgroup of $G$ (with $G$ simple), then $H \cap H^x = 1$ and so \cite[Conjecture A]{LS} immediately implies Conjecture \ref{c:nilp}. The Lisi-Sabatini conjecture is proved in \cite{LS} for all metanilpotent groups of odd order, as well as all sufficiently large symmetric and alternating groups. The latter asymptotic result relies on a probabilistic argument, which applies recent work of Diaconis et al. \cite{D25} and Eberhard \cite{Eb}. 
\end{remk}

\begin{remk}\label{r:ab}
Let $G$ be a finite non-abelian simple group and let $H$ be a nilpotent subgroup. Recall that Conjecture \ref{c:nilp} asserts that $H \cap H^x = 1$ for some $x \in G$. In fact, in several special cases, it has been shown that if $A$ and $B$ are any nilpotent subgroups of $G$, then $A \cap B^x = 1$ for some $x \in G$. This is proved in 
\cite{Z_spor,Z_sym} for alternating and sporadic groups, and we refer the reader to
\cite{Z_dim3,Z_L2} for the low-rank Lie type groups ${\rm L}_2(q), {\rm L}_3(q)$ and ${\rm U}_3(q)$. Given these results, it seems reasonable to speculate that this property holds for all simple groups, which can be viewed as a natural generalisation of Conjecture \ref{c:nilp}.
\end{remk}

\begin{remk}\label{r:depth}
Let $G$ be a finite group and let $H$ be a proper subgroup. In 2011, Boltje, Danz and K\"{u}lshammer \cite{BDK} introduced and studied the \emph{depth} of $H$, denoted $d_G(H)$, which is a positive integer defined in terms of the inclusion of complex group algebras $\mathbb{C}H \subseteq \mathbb{C}G$ (we refer the reader to \cite{BDK} for the formal definition). This notion has been the focus of numerous papers and there  has been a particular interest in studying the depth of subgroups of simple and almost simple groups (see \cite{B25} and the references therein). In this setting, if $G$ is almost simple and $H \ne 1$ is core-free, then we have 
\begin{equation}\label{e:dep}
3 \leqs d_G(H) \leqs 2b(G,H) -1
\end{equation}
(see \cite[Propositions 2.2 and 2.6]{B25}, for example). In particular, if $b(G,H) = 2$ then $d_G(H) = 3$, so the main theorem of \cite{ZM} implies that every (non-trivial) Sylow subgroup of a simple group has depth $3$ (and Conjecture \ref{c:nilp} asserts that every non-trivial nilpotent subgroup has depth $3$). In view of Theorem \ref{t:main1}, in order to determine the depth of every Sylow subgroup of an almost simple group, it just remains to consider the pairs $(G,H)$ recorded in Table \ref{tab:zen}. Since $b(G,H) = 3$, the bounds in \eqref{e:dep} imply that $d_G(H) \in \{3,4,5\}$ and further work is needed in order to determine the precise depth of $H$:

\begin{itemize}\addtolength{\itemsep}{0.2\baselineskip}
\item[{\rm (a)}] As noted in \cite[Remark 8(b)]{B25}, if $G = {\rm PGL}_2(q)$ and $H$ is a Sylow $2$-subgroup of $G$, where $q$ is a Mersenne prime, then $d_G(H) = 5$. 
\item[{\rm (b)}] The relevant groups with socle ${\rm P\O}_8^{+}(3)$, ${\rm L}_4(2)$, ${\rm L}_3(4)$ and ${\rm L}_2(9)$ can be handled using {\sc Magma} and the character-theoretic approach for calculating $d_G(H)$, which is explained in \cite[Section 2.1.1]{B25}. In every case, we get $d_G(H) = 5$.
\item[{\rm (c)}] It remains an open problem to determine $d_G(H)$ for the remaining cases in Table \ref{tab:zen}, where $p=2$ and $G = F_4(2).2$, $E_6(2).2$, ${\rm L}_n(2).2$ or ${\rm O}_{n}^{+}(2)$. We have used {\sc Magma} to check that $d_G(H) = 5$ for $G = {\rm L}_n(2).2$ with $3 \leqs n \leqs 7$, and for $G = {\rm O}_n^{+}(2)$ with $n = 8$ or $10$. 
\end{itemize}

\noindent So there is some evidence to suggest that if $G$ is almost simple and $H$ is a Sylow $p$-subgroup of $G$, where $p$ is a prime divisor of $|G|$, then either $d_G(H) = 3$, or $(G,H)$ is one of the cases in Table \ref{tab:zen} and $d_G(H) = 5$.
\end{remk}

\begin{remk}\label{r:sylow}
Let $G$ be an almost simple group and set $N = N_G(H)$, where $H$ is a Sylow $p$-subgroup of $G$ and $p$ is a prime divisor of $|G|$. Then we can identify $G/N$ with the set of Sylow $p$-subgroups of $G$ and we can think of $b(G,N)$ as the base size with respect to the natural transitive action of $G$ by conjugation on the set of Sylow $p$-subgroups of $G$. Of course, if $N = H$, then $b(G,N)$ is given in Theorem \ref{t:main1}. However, it remains an open problem to determine $b(G,N)$ when $N \ne H$. Here we speculate that the upper bound $b(G,N) \leqs 4$ is best possible, noting that equality holds when $G = S_6$ and $p=3$ (in which case, $N = S_3 \wr S_2$ is a maximal subgroup of $G$). Let us also observe that there are examples where $b(G,N) > 2$ and $p$ is arbitrarily large (in contrast to the situation for $b(G,H)$, where Theorem \ref{t:zenkov} shows that $b(G,H) >2$ only if $p=2$ or $3$). For example, if $G = {\rm L}_2(p)$ with $p \geqs 5$, then $N$ is a Borel subgroup of $G$ and we have $b(G,N) = 3$ (see \cite[Proposition 4.1]{Bur21}).
\end{remk}

\vs

\noindent \textbf{Note added in press.} In a recent preprint \cite{BH2}, we have used a probabilistic approach to prove the Lisi-Sabatini conjecture  for all non-alternating simple groups (see Remark \ref{r:ls}). And by combining the main result in \cite{BH2} with earlier work of Kurmazov \cite{Kurm}, we are able to complete the proof of Vdovin's conjecture on nilpotent subgroups of simple groups (see Conjecture \ref{c:nilp}). In fact, we have established the stronger form of the conjecture alluded to in Remark \ref{r:ab}: if $G$ is simple and $A,B$ are nilpotent subgroups of $G$, then $A \cap B^x = 1$ for some $x \in G$.

\vs

\noindent \textbf{Notation.} For a finite group $G$ and positive integer $n$, we write $C_n$, or just $n$, for a cyclic group of order $n$ and $G^n$ for the direct product of $n$ copies of $G$. An unspecified extension of $G$ by a group $H$ will be denoted by $G.H$; if the extension splits then we may write $G{:}H$. If $X$ is a subset of $G$, then we will write $i_2(X)$ for the number of involutions in $X$. And if $p$ is a prime, then ${\rm Syl}_p(G)$ is the set of Sylow $p$-subgroups of $G$. We will write $[n]$ for an unspecified soluble group of order $n$. Throughout the paper, we adopt the standard notation for simple groups of Lie type from \cite{KL} and we write $(a,b)$ for the highest common factor of the positive integers $a$ and $b$.

\vs

\noindent \textbf{Acknowledgements.} The second author thanks the London Mathematical Society for their support as an LMS Early Career Research Fellow at the University of St Andrews.

\section{Preliminaries}\label{s:prel}

In this section we record several preliminary results, which will be needed in the proof of Theorem \ref{t:main1}. We begin by briefly recalling the probabilistic method for bounding the base size of a finite permutation group, which is at the heart of our proof.

\subsection{Bases}\label{ss:bases}

Let $G \leqs {\rm Sym}(\O)$ be a transitive permutation group on a finite set $\O$ with point stabiliser $H \ne 1$. Recall that a subset $B$ of $\O$ is a \emph{base} for $G$ if the pointwise stabiliser of $B$ in $G$ is trivial. We write $b(G,H)$ for the minimal size of a base, which we refer to as the \emph{base size} of $G$. Notice that this coincides with the minimal number of conjugates of $H$ with trivial intersection.

Next let 
\begin{equation}\label{e:Q}
Q(G,H) = \frac{|\{(\a,\b) \in \O^2 \,:\, G_{\a} \cap G_{\b} \ne 1\}|}{|\O|^2}
\end{equation}
be the probability that a random pair of points in $\O$ is not a base for $G$, which means that $b(G,H) = 2$ if and only if $Q(G,H) < 1$. Then as explained in the proof of \cite[Theorem 1.3]{LSh99} (see \cite[Section 2]{LSh99}), we have 
\begin{equation}\label{e:Qhat}
Q(G,H) \leqs \widehat{Q}(G,H) = \sum_{i=1}^k|x_i^G| \cdot \left(\frac{|x_i^G \cap H|}{|x_i^G|}\right)^2,
\end{equation}
where $x_1, \ldots, x_k$ is a complete set of representatives of the conjugacy classes in $G$ of elements of prime order. The following result is an immediate consequence.

\begin{lem}\label{l:prob}
If $\widehat{Q}(G,H) < 1$ then $b(G,H) = 2$.
\end{lem}

\begin{rem}
In the proof of Theorem \ref{t:main1}, we are interested in the case where $H$ is a Sylow $2$-subgroup of $G$. Since $|x^G \cap H| = 0$ for all $x \in G$ of odd order, we have    
\[
\widehat{Q}(G,H) = \sum_{i=1}^{\ell}|y_i^G| \cdot \left(\frac{|y_i^G \cap H|}{|y_i^G|}\right)^2,
\]
where $y_1, \ldots, y_{\ell}$ is a set of representatives of the conjugacy classes of involutions in $G$.
\end{rem}

The next elementary lemma is useful for obtaining an upper bound on 
$\widehat{Q}(G,H)$.

\begin{lem}\label{l:est}
Suppose $x_1, \ldots, x_m$ represent distinct $G$-classes such that $\sum_i |x_i^G \cap H| \leqs a$ and $|x_i^G| \geqs b$ for all $i$. Then
\[
\sum_{i=1}^m|x_i^G| \cdot \left(\frac{|x_i^G \cap H|}{|x_i^G|}\right)^2 \leqs a^2b^{-1}.
\]
\end{lem}

\begin{proof}
See \cite[Lemma 2.1]{B07}.
\end{proof}

\subsection{Sylow subgroups}\label{ss:sylow}

We will need information of the orders of Sylow $2$-subgroups in almost simple groups of Lie type in odd characteristic. To this end, the following number-theoretic result will be useful. Here and throughout the paper, we write $(m)_2$ for the largest power of $2$ dividing the positive integer $m$ (so for example, $(24)_2 = 8$). 

\begin{lem}\label{l:nt}
Let $q$ be an odd prime power, let $d$ be a positive integer and let $\e \in \{-1,1\}$. Then
\[
(q^d -\e)_2 = \left\{\begin{array}{ll}
(q -\e)_2 & \mbox{if $d$ is odd} \\
(q^2-1)_2(d/2)_2 & \mbox{if $d$ is even and $\e = 1$} \\ 
2 & \mbox{otherwise.}
\end{array}\right.
\] 
\end{lem}

\begin{proof}
This is \cite[Lemma 2.1(i)]{BBGT}.
\end{proof}

We will also need the following elementary result concerning factorials.

\begin{lem}\label{l:factorial}
Let $m$ be a positive integer. Then $2^m(m!)_2 = ((2m)!)_2$ and $(m!)_2 < 2^m$.
\end{lem}

\begin{proof}
The first statement is trivial. And for the inequality, set $(m!)_2 = 2^{\ell}$ and note that 
\[
\ell = \sum_{i=1}^{\infty}\left\lfloor \frac{m}{2^i} \right\rfloor < m \sum_{i=1}^{\infty}2^{-i} = m.
\]
The result follows.
\end{proof}

Our main result is the following.

\begin{prop}\label{p:syl}
Let $G$ be an almost simple group of Lie type over $\mathbb{F}_q$ with socle $G_0$ and assume $q$ is odd. Let $H$ be a Sylow $2$-subgroup of $G$ and set $H_0 = H \cap G_0$. Then we have $|H| = |H_0||G:G_0|_2$, where $|H_0|$ is recorded in Table \ref{tab:syl}.
\end{prop}

\begin{proof}
This is a routine exercise, working with Lemma \ref{l:nt} and the order of $G_0$, which is conveniently recorded in \cite[Tables 5.1.A, 5.1.B]{KL}. 
\end{proof}

\begin{table}
\[
\begin{array}{lll} \hline
G_0 & |H_0| & \mbox{Conditions} \\ \hline
{\rm L}_n^{\e}(q) & \a^m((q-\e)_2)^m(m!)_2 & n = 2m+1 \geqs 3 \\
& \frac{1}{d_2}\a^m((q-\e)_2)^{m-1}(m!)_2 & n = 2m \geqs 2,\, d = (n,q-\e) \\
{\rm PSp}_n(q) & \frac{1}{2}\a^m(m!)_2 & n = 2m \geqs 2 \\
\O_n(q) & \frac{1}{2}\a^m(m!)_2 & n = 2m+1 \geqs 3 \\ 
{\rm P\O}_n^{+}(q) & \frac{1}{2d}\a^m(m!)_2 & \mbox{$n = 2m$, $m \geqs 4$ even, $d = (4,q^m-1)$} \\
& \frac{1}{d}\a^{m-1}(q-1)_2((m-1)!)_2 & \mbox{$n = 2m$, $m \geqs 3$ odd, $d = (4,q^m-1)$} \\
{\rm P\O}_n^{-}(q) & \a^{m-1}((m-1)!)_2 & \mbox{$n = 2m$, $m \geqs 4$ even} \\
& \frac{1}{d}\a^{m-1}(q+1)_2((m-1)!)_2 & \mbox{$n = 2m$, $m \geqs 3$ odd, $d = (4,q^m+1)$} \\
E_8(q) & 2^6\a^8 & \\
E_7(q) & 2^2\a^7 & \\
E_6^{\e}(q) & 2^3\a^4((q-\e)_2)^2 & \\
F_4(q) & 2^3\a^4 & \\
{}^3D_4(q) & \a^2 & \\
G_2(q),\, q \geqs 5 & \a^2 & \\
{}^2G_2(q),\, q \geqs 27 & \a & \\ \hline
\end{array}
\]
\caption{The order of $H_0 \in {\rm Syl}_2(G_0)$ with $q$ odd, $\a = (q^2-1)_2$}
\label{tab:syl}
\end{table}

\subsection{Involutions}\label{ss:invol}

Our proof of Theorem \ref{t:main1} will rely on Lemma \ref{l:prob}, which means that we will need to compute an appropriate upper bound on $\widehat{Q}(G,H)$. And to do this, we will make use of Lemma \ref{l:est}. Let $\mathcal{I} = \{y_1, \ldots, y_{\ell}\}$ be a set of representatives of the conjugacy classes in $G$ of involutions.

Typically, we will proceed by partitioning $\mathcal{I}$ into two or three disjoint subsets, say 
\[
\mbox{$\mathcal{I} = \mathcal{I}_1 \cup \cdots \cup \mathcal{I}_m$ with 
$\mathcal{I}_j = \{y_{j,1}, \ldots, y_{j,k_j}\}$,}
\] 
 and for each $j$ we will establish bounds of the form 
\[
\sum_{i=1}^{k_j} |y_{j,i}^G \cap H| \leqs a_j,\;\; \min\{|y_{j,i}^G| \,:\, 1 \leqs i \leqs k_j\} \geqs b_j.
\]
Then Lemma \ref{l:est} yields
\[
\widehat{Q}(G,H) \leqs \sum_{j=1}^{m} a_j^2b_j^{-1}
\]
and our aim is to use this upper bound to force $\widehat{Q}(G,H)<1$.

In order to effectively apply this approach, we need some information on the `small' conjugacy classes of involutions in almost simple groups of Lie type in odd characteristic. For the exceptional groups, it turns out that we will only require a lower bound on $|x^G|$ that is valid for every involution $x\in G$ and we can refer to \cite[Proposition 2.11]{BTh} for this information (see \eqref{e:fq} in the proof of Proposition \ref{p:ex}). However, for classical groups, we will usually need to know the sizes of the two or three smallest conjugacy classes and with this aim in mind we present the following technical result. In the statement, we refer to the notation for involutions from \cite[Table 4.5.1]{GLS}.

\begin{prop}\label{p:inv}
Let $G$ be an almost simple classical group over $\mathbb{F}_q$ with socle $G_0$, where $q$ is odd. Let $x \in G$ be an involution.

\begin{itemize}\addtolength{\itemsep}{0.2\baselineskip}
\item[{\rm (i)}] If $G_0 = {\rm L}_n^{\e}(q)$ with $n \geqs 5$, then either $x$ is a $t_1$-type involution and
\[
|x^G| = \frac{|{\rm GL}^{\e}_n(q)|}{|{\rm GL}^{\e}_{n-1}(q)||{\rm GL}^{\e}_1(q)|} = \frac{q^{n-1}(q^n-\e^n)}{q-\e},
\]
or $|x^G| \geqs f(n,q)$ with
\[
f(n,q) = \left\{
\begin{array}{ll}
\frac{|G_0|}{|{\rm SU}_5(q_0)|} =  \frac{1}{d}q^5(q+1)(q^{3/2}-1)(q^2+1)(q^{5/2}-1) 
& \mbox{if $(\e,n,q) = (+,5,q_0^2)$} \\
\frac{|G_0|}{|{\rm PSp}_6(q)|} = \frac{2}{d}q^6(q^3-\e)(q^5-\e) & \mbox{if $n=6$} \\
\frac{|{\rm GL}_n^{\e}(q)|}{|{\rm GL}_{n-2}^{\e}(q)||{\rm GL}_2^{\e}(q)|} = \frac{q^{2n-4}(q^{n-1}-\e^{n-1})(q^n-\e^n)}{(q^2-1)(q-\e)} & \mbox{otherwise,}
\end{array}\right.
\]
where $d = (n,q-\e)$.

\item[{\rm (ii)}] If $G_0 = {\rm L}_4^{\e}(q)$, then either $x$ is a $\gamma_1$ graph automorphism and
\[
|x^G| \geqs \frac{|G_0|}{|{\rm PSp}_4(q)|} \geqs \frac{1}{2}q^2(q^3-\e),
\]
or 
\[
|x^G| \geqs \frac{|{\rm GL}^{\e}_4(q)|}{|{\rm GL}^{\e}_{3}(q)||{\rm GL}^{\e}_1(q)|} = \frac{q^{3}(q^4-1)}{q-\e}.
\]

\item[{\rm (iii)}] If $G_0 = {\rm PSp}_n(q)$ with $n \geqs 6$, then either $x$ is a $t_1$-type involution and
\[
|x^G| = \frac{|{\rm Sp}_n(q)|}{|{\rm Sp}_{n-2}(q)||{\rm Sp}_2(q)|} = \frac{q^{n-2}(q^n-1)}{q^2-1},
\]
or $|x^G| \geqs f(n,q)$ with
\[
f(n,q) = \left\{
\begin{array}{ll}
\frac{|G_0|}{|{\rm Sp}_{6}(q_0)|} = \frac{1}{2}q^{9/2}(q+1)(q^2+1)(q^3+1) & \mbox{if $(n,q)=(6,q_0^2)$} \\
\frac{|G_0|}{|{\rm GU}_{3}(q)|} = \frac{1}{2}q^6(q-1)(q^2+1)(q^3-1) & \mbox{if $n=6$, $q \ne q_0^2$} \\
\frac{|G_0|}{|{\rm Sp}_{4}(q^2)|} = \frac{1}{2}q^8(q^2-1)(q^6-1) & \mbox{if $n=8$} \\
\frac{|{\rm Sp}_n(q)|}{|{\rm Sp}_{n-4}(q)||{\rm Sp}_4(q)|} = \frac{q^{2n-4}(q^{n-2}-1)(q^n-1)}{(q^4-1)(q^2-1)} & \mbox{otherwise.} \\
\end{array}\right.
\]

\item[{\rm (iv)}] If $G_0 = {\rm PSp}_4(q)$, then either $x$ is a $t_2$-type or $t_2'$-type involution and
\[
|x^G| \geqs \frac{|G_0|}{|{\rm Sp}_{2}(q^2)|} = \frac{1}{2}q^2(q^2-1),
\]
or $|x^G| \geqs f(n,q)$ with
\[
f(n,q) = \left\{
\begin{array}{ll}
\frac{|G_0|}{|{\rm Sp}_{4}(q_0)|} = \frac{1}{2}q^2(q+1)(q^2+1) & \mbox{if $q=q_0^2$} \\
\frac{|G_0|}{|{\rm GU}_{2}(q)|} = \frac{1}{2}q^3(q-1)(q^2+1) & \mbox{otherwise.} \\
\end{array}\right.
\]

\item[{\rm (v)}] If $G_0 = \O_n(q)$ with $n = 2m+1 \geqs 7$, then either $x$ is a $t_m$-type or $t_m'$-type involution and
\[
|x^G| \geqs \frac{|{\rm SO}_n(q)|}{2|{\rm SO}_{n-1}^{-}(q)|} = \frac{1}{2}q^{m}(q^{m}-1),
\]
or 
\[
|x^G| \geqs \frac{|{\rm SO}_n(q)|}{2|{\rm SO}_{n-2}(q)||{\rm SO}_{2}^{-}(q)|} = \frac{q^{n-2}(q^{n-1}-1)}{2(q+1)}.
\]

\item[{\rm (vi)}] If $G_0 = {\rm P\O}_n^{\e}(q)$ with $n = 2m \geqs 8$, then either $x$ is a $\gamma_1$ graph automorphism and
\[
|x^G| \geqs \frac{|{\rm SO}_n^{\e}(q)|}{2|{\rm SO}_{n-1}(q)|} = \frac{1}{2}q^{m-1}(q^{m}-1),
\]
or
\[
|x^G| \geqs \frac{|{\rm SO}_n^{\e}(q)|}{2|{\rm SO}_{n-2}^{-\e}(q)||{\rm SO}_{2}^{-}(q)|} = \frac{q^{n-2}(q^m-\e)(q^{m-1}-1)}{2(q+1)}.
\]
\end{itemize}
\end{prop}

\begin{proof}
This is a straightforward exercise, working with the detailed information on conjugacy classes of semisimple involutions and involutory graph automorphisms presented in \cite[Table 4.5.1]{GLS} (also see \cite[Chapter 3]{BGiu}).

For example, suppose $G_0 = {\rm PSp}_n(q)$ with $n=2m \geqs 10$ and let $x \in G$ be an involution. If $x \in G$ is a field automorphism, then $q = q_0^2$ and \cite[Proposition 3.4.15]{BGiu} gives
\[
|x^G| \geqs \frac{|G_0|}{|{\rm Sp}_n(q_0)|} > \frac{1}{2}q^{n(n+1)/4}.
\]
Otherwise, $x \in {\rm PGSp}_n(q)$ is semisimple and we can read off $|x^G|$ from \cite[Table 4.5.1]{GLS} (or \cite[Table B.5]{BGiu}). Then either $x$ is an involution of type $t_m$, $t_m'$, $t_{m/2}$ or $t_{m/2}'$ (in the notation of \cite[Table 4.5.1]{GLS}), in which case
\[
|x^G| \geqs \frac{|G_0|}{|{\rm Sp}_m(q^2)|} > \frac{1}{4}q^{n^2/4},
\]
or $x$ is an involution of type $t_i$ with $1 \leqs i < m/2$ and we have
\[
|x^G| = \frac{|{\rm Sp}_n(q)|}{|{\rm Sp}_{n-2i}(q)||{\rm Sp}_{2i}(q)|}.
\]
By comparing the above expressions, we deduce that $|x^G|$ is minimal when $x$ is a $t_1$-type involution. And for all other involutions, $|x^G|$ is minimal when $x$ is of type $t_2$, which gives the result recorded in part (iii). Note that the latter statement is false when $n = 8$ since 
\[
\frac{|{\rm Sp}_8(q)|}{2|{\rm Sp}_{4}(q)|^2} > \frac{|{\rm Sp}_8(q)|}{2|{\rm Sp}_{4}(q^2)|},
\]
which explains why the expression for $f(n,q)$ in (iii) has a different form when $n=8$ (and similarly when $n=6$).

A very similar argument applies in each of the remaining cases and we omit the details.
\end{proof}
 
\subsection{Computational methods}\label{ss:comp}

In the proof of Theorem \ref{t:main1}, it will be convenient to use direct computation to handle certain low rank groups of Lie type defined over small fields. For all the computations, we will use {\sc Magma} \cite{magma} (version V2.28-21).

In order to state our main result, let us define $\mathcal{A}$ to be the following collection of simple groups of Lie type, recalling that we use the notation for simple groups from \cite{KL}:

\renewcommand{\arraystretch}{1.2}
\[
\begin{array}{l}
{\rm L}_4^{\e}(3), \, {\rm L}_4^{\e}(5), \, {\rm L}_4^{\e}(7), \, {\rm L}_4^{\e}(9), \, {\rm L}_4^{\e}(17), \,{\rm L}_4(31), \, {\rm L}_4(127), \, {\rm U}_5(3), \, {\rm L}_5(9), \, {\rm L}_6^{\e}(3), \, {\rm L}_6(9), \, {\rm L}_8^{\e}(3), \\
{\rm PSp}_{4}(3), \, {\rm PSp}_{4}(5), \, {\rm PSp}_{4}(7),\, {\rm PSp}_{4}(9), \, {\rm PSp}_{4}(17), \, {\rm PSp}_{6}(3), \, {\rm PSp}_{8}(3), {\rm PSp}_{10}(3), \\
\O_7(3), \, \O_9(3), \, \O_{11}(3), \, \O_{13}(3), \, {\rm P\O}_{8}^{\pm}(3), \, {\rm P\O}_{10}^{\pm}(3),\, {\rm P\O}_{12}^{\pm}(3), \, G_2(3)
\end{array}
\]
\renewcommand{\arraystretch}{1}

\begin{prop}\label{p:comp}
Let $G$ be an almost simple group with socle $G_0 \in \mathcal{A}$ and let $H$ be a Sylow $2$-subgroup of $G$. Then $b(G,H) = 2$.
\end{prop}

\begin{proof}
To begin with, assume $G_0 \ne {\rm P\O}_{8}^{+}(3)$. We first construct the full automorphism group $L = {\rm Aut}(G_0)$ as a permutation group, using the function \texttt{AutomorphismGroupSimpleGroup}, and we take a Sylow $2$-subgroup $J$ of $L$. In every case, we note that $L/G_0$ is a $2$-group and we can use random search to find an element $x \in L$ such that $J \cap J^x = 1$. This means that $J$ has a regular orbit on $L/J$, which in turn implies that $H$ has a regular orbit on $G/H$ and thus $b(G,H) = 2$.

Finally, suppose $G_0 = {\rm P\O}_{8}^{+}(3)$. Here $|{\rm Out}(G_0)| = 24$ and we can use {\sc Magma} to construct the index-three subgroup $L = {\rm PGO}_8^{+}(q) = G_0.D_8$ of ${\rm Aut}(G_0)$ as a permutation group of degree $3360$. We can then proceed as before, taking a Sylow $2$-subgroup $J$ of $L$ and using random search to find an element $x \in L$ such that $J \cap J^x = 1$. The result follows.
\end{proof}

\section{Proof of Theorem \ref{t:main1}}\label{s:main1}

We are now ready to prove Theorem \ref{t:main1}. Our starting point is Zenkov's main reduction theorem for almost simple groups from \cite{Z_92p}, which we stated as Theorem \ref{t:zenkov} in Section \ref{s:intro}. For the reader's convenience, in the following remark we provide a brief overview of Zenkov's proof.
 
\begin{rem}\label{r:zenkov}
Let $G$ be an almost simple group with socle $G_0$ and let $H$ be a Sylow $p$-subgroup of $G$, where $p$ is a prime divisor of $|G|$. 

\begin{itemize}\addtolength{\itemsep}{0.2\baselineskip}
\item[{\rm (a)}] A starting point is the main theorem from \cite{ZM} on simple groups, which allows us to assume that $p$ divides $|G:G_0|$. In particular, $p=2$ when $G_0$ is an alternating or sporadic group; the argument for symmetric groups is given in \cite{ZM} and the $12$ sporadic groups of the form $G = G_0.2$ are handled case-by-case in \cite[Section 3]{Z_92p} (the result for sporadic groups can also be verified very easily using {\sc Magma} \cite{magma}, proceeding as in the proof of Proposition \ref{p:comp}).

\item[{\rm (b)}] Now assume $G_0$ is a simple group of Lie type over $\mathbb{F}_q$. For $p$ odd, a long and technical argument, based on the elimination of a counterexample of minimal order, is presented in \cite[Section 6]{Z_92p}. Here the main result is \cite[Theorem 6.1]{Z_92p}, which shows that $b(G,H) = 3$ if and only if $p=3$ and $G_0 = {\rm P\O}_8^{+}(3)$, with $G$ recorded in Table \ref{tab:zen}. It is worth noting that it is easy to use {\sc Magma} to verify that $b(G,H) = 3$ for the two relevant groups with $G_0 = {\rm P\O}_8^{+}(3)$.

\item[{\rm (c)}] The analysis for $p=2$ divides naturally into two cases, according to whether or not $q$ is odd or even, and the main details are given in Sections 4 and 5 of \cite{Z_92p}, respectively. For $q$ odd, the main result is \cite[Theorem 4.1]{Z_92p}, which implies that $b(G,H) = 2$ unless $q=9$ or $q$ is a Mersenne or Fermat prime. The proof relies on a detailed analysis of a minimal counterexample, with \cite[Lemma 4.3]{Z_92p} as a key tool. The latter states that if $G$ is a minimal counterexample, then $O_2(C) \cap H^x \ne 1$ for all $x \in G$, where $C = C_G(z)$ is the centraliser of an involution $z \in Z(H) \cap G_0$. Then by studying the structure of involution centralisers in $G$, in many cases Zenkov is able to reach a contradiction by establishing the existence of an element $x \in G$ with $O_2(C) \cap H^x = 1$.

\item[{\rm (d)}] In \cite[Lemma 3.18]{Z_92p}, Zenkov proves that $b(G,H) = 3$ when $p=2$ and either $G = {\rm Aut}({\rm L}_2(9)) = {\rm L}_2(9).2^2 \cong A_6.2^2$ or $G = {\rm PGL}_2(q)$ with $q \geqs 7$ a Mersenne prime.

\item[{\rm (e)}] Finally, suppose $p=2$ and $q$ is even. Here the first main result is \cite[Theorem 5.1]{Z_92p}, which tells us that $b(G,H) = 3$ only if $G_0$ is one of the following (here we exclude ${\rm L}_3(2).2 \cong {\rm PGL}_2(7)$, ${\rm L}_4(2).2 \cong S_8$ and ${\rm PSp}_4(2)'.2^2 \cong {\rm L}_2(9).2^2$):
\[
\hspace{12mm} {\rm L}_3(4), \, {\rm L}_n(2) \mbox{ $(n \geqs 5)$}, \, \O_n^{\pm}(2) \mbox{ $(n \geqs 8)$},\, E_6(2), \, F_4(2), \, {}^2F_4(2)', \, G_2(2)'.
\]
As before, the proof proceeds by studying the structure of a minimal counterexample, using the fact that $N_G(H)$ is a Borel subgroup of $G$. 

\vspace{2mm}

\noindent And then it remains to determine which of these possibilities lead to the genuine examples with $b(G,H) = 3$ recorded in Table \ref{tab:zen}. For instance,  suppose $G = {\rm L}_n(2).2$, $\Omega_n^+(2).2$, $E_6(2).2$ or $F_4(2).2$ and set $H_0 = H \cap G_0$. By \cite[Lemma 3.13]{Z_92p}, $H_0$ has a  unique regular orbit on ${\rm Syl}_2(G_0)$ with respect to the usual conjugation action. And since $N_{G_0}(H_0) = H_0$, it follows that $H_0$ has a unique regular orbit on $G_0/H_0$ and this implies that $b(G,H)>2$ (indeed, if $b(G,H) = 2$ then $H_0$ would have at least $|H:H_0| = 2$ regular orbits on $G_0/H_0$).
\end{itemize}
\end{rem}

By combining Zenkov's main theorem for almost simple groups in \cite{Z_92p} (see Theorem \ref{t:zenkov} in Section \ref{s:intro}) with \cite[Lemma 3.18]{Z_92p} (socle ${\rm L}_2(q)$) and the main results in \cite{Z_dim3} (socle ${\rm L}_3(q)$ or ${\rm U}_3(q)$) and \cite{ZM} (simple groups), we see that it suffices to prove the following result.

\begin{thm}\label{t:main2}
Let $G$ be an almost simple group of Lie type over $\mathbb{F}_q$ with socle $G_0$ and let $H$ be a Sylow $2$-subgroup of $G$. Assume that 
\begin{itemize}\addtolength{\itemsep}{0.2\baselineskip}
\item[{\rm (i)}] $G_0$ is not isomorphic to ${\rm L}_2(q)$, ${\rm L}_3(q)$ or ${\rm U}_3(q)$;
\item[{\rm (ii)}] $|G:G_0|$ is even; and 
\item[{\rm (iii)}] Either $q=9$, or $q$ is a Mersenne or Fermat prime.
\end{itemize}
Then $b(G,H) = 2$.  
\end{thm}

\subsection{Exceptional groups}\label{ss:excep}

\begin{prop}\label{p:ex}
The conclusion to Theorem \ref{t:main2} holds when $G$ is an exceptional group of Lie type.
\end{prop}

\begin{proof}
First observe that the hypotheses in parts (ii) and (iii) of Theorem \ref{t:main2} imply that \[
G_0 \ne {}^2B_2(q), \, {}^2G_2(q)', \, {}^2F_4(q)'.
\]
Set $\a = (q^2-1)_2$ and $\b = \delta_{9,q}$ (so $\b = 1$ if $q=9$, otherwise $\b=0$). Note that $\a \leqs 2(q+1)$.

\renewcommand{\arraystretch}{1.2}
Let $x \in G$ be an involution. Then by appealing to \cite[Proposition 2.11]{BTh}, we see that $|x^G|> f_G(q)$, where $f_G(q)$ is defined as follows (with $\gamma = (q-1)/q$ and $\delta = (3,q-\e)^{-1}$):
\begin{equation}\label{e:fq}
\begin{array}{l|llllll}
G_0 & E_8(q) & E_7(q) & E_6^{\e}(q) & F_4(q) & {}^3D_4(q) & G_2(q) \\ \hline
f_G(q) & q^{112} & \frac{1}{2}\gamma q^{54} & \gamma\delta q^{26} & q^{16} & q^{14} & q^7
\end{array}
\end{equation}
Let $H$ be a Sylow $2$-subgroup of $G$ and set $H_0 = H \cap G_0$. As in \eqref{e:Q}, let $Q(G,H)$ be the probability that two random elements in $\O = G/H$ form a base for $G$ and recall that $Q(G,H) \leqs \widehat{Q}(G,H)$ as in \eqref{e:Qhat}, so it suffices to show that $\widehat{Q}(G,H)<1$.  
\renewcommand{\arraystretch}{1}

First assume $G_0 = E_8(q)$. Here Proposition \ref{p:syl} gives $|H_0| = 2^6\a^8$ and we deduce that $|H| \leqs 2^{14+\b}(q+1)^8 = a$. Since $|x^G| > f_G(q)$ for every involution $x \in G$, Lemma \ref{l:est} implies that
\begin{equation}\label{e:easy}
\widehat{Q}(G,H) \leqs a^2f_G(q)^{-1}
\end{equation}
and it is easy to check that this upper bound is less than $1$ for all $q \geqs 3$. Similarly, \eqref{e:easy} holds in each of the remaining cases, where $f_G(q)$ is given in \eqref{e:fq} and $a = 2^{c+\b}(q+1)^d$ with  
\[
(c,d) = (10,7),\, (8,6), \, (7,4), \, (2,2), \, (3,2) 
\]
for $G_0 = E_7(q), E_6^{\e}(q), F_4(q), {}^3D_4(q), G_2(q)$, respectively. It is now a routine exercise to check that the upper bound in \eqref{e:easy} is sufficient unless $G_0 = F_4(3)$, $G_2(5)$ or $G_2(3)$. Here the first two possibilities can be discarded since $G$ is simple (recall that we may assume $|G:G_0|$ is even), while the latter case was handled in Proposition \ref{p:comp}.
\end{proof}

\subsection{Classical groups}\label{ss:class}

In this section we complete the proof of Theorem \ref{t:main1} by establishing Theorem \ref{t:main2} for classical groups. Throughout, we will adopt the standard notation for classical groups given in \cite{KL}. In particular, we denote the simple orthogonal groups by ${\rm P\O}_n^{\e}(q)$, which differs from the notation for orthogonal groups used in the Atlas \cite{Atlas}, for example. 

Let $G$ be an almost simple classical group over $\mathbb{F}_q$ as in the statement of Theorem \ref{t:main2}. By combining the hypotheses of the theorem with the existence of well-known isomorphisms among some of the low-dimensional classical groups (see \cite[Proposition 2.9.1]{KL}, for example), we may assume $G_0$ is one of the following:
\[
\begin{array}{ll}
{\rm L}_n(q), \, {\rm U}_n(q) & n \geqs 4 \\
{\rm PSp}_n(q) & \mbox{$n \geqs 4$ even} \\
\O_n(q) & \mbox{$n \geqs 7$ odd, $q$ odd} \\
{\rm P\O}_n^{+}(q), \, {\rm P\O}_n^{-}(q) & \mbox{$n \geqs 8$ even}
\end{array}
\]

We will sometimes refer to the `type' of a subgroup $L$ of $G$, which is intended to provide an approximate description of the structure of $L$. For example, if $G_0 = {\rm P\O}_{10}^{+}(q)$ then we may refer to a subgroup $L$ of type $({\rm O}_{2}^{-}(q) \wr S_4) \perp {\rm O}_{2}^{+}(q)$, which indicates that $L$ is the stabiliser in $G$ of an orthogonal decomposition 
\[
V = U \perp W = (U_1 \perp U_2 \perp U_3 \perp U_4) \perp W
\]
of the natural module $V$ for $G_0$, where each $U_i$ is a $2$-dimensional non-degenerate subspace of minus-type, and $W$ is a non-degenerate $2$-space of plus-type (recall that if $m = 2\ell$ is even, then an orthogonal $m$-space is of \emph{plus-type} if it contains a totally singular $\ell$-space, otherwise it is a \emph{minus-type} space).

Finally, we refer the reader to \cite[Table 4.5.1]{GLS} and \cite[Chapter 3]{BGiu} for detailed information on the conjugacy classes of involutions in ${\rm Aut}(G_0)$. As in the statement of Proposition \ref{p:inv}, we will freely adopt the notation for involutions given in \cite[Table 4.5.1]{GLS}.

\begin{prop}\label{p:psl}
The conclusion to Theorem \ref{t:main2} holds if $G_0 = {\rm L}_n(q)$.
\end{prop}

\begin{proof}
Let $H$ be a Sylow $2$-subgroup of $G$ and set $H_0 = H \cap G_0$ and $\a = (q^2-1)_2$ as before. In addition, set $d = (n,q-1)$ and $\b = \delta_{9,q}$.

Given an integer $1 \leqs k \leqs n/2$, we say that an involution $x \in G$ is of type $t_k$ if it is conjugate in ${\rm PGL}_n(q)$ to $y = \hat{y}Z$, where $\hat{y} = (-I_k,I_{n-k}) \in {\rm GL}_n(q)$ is a diagonal matrix and $Z$ is the centre of ${\rm GL}_n(q)$. This notation is consistent with \cite[Table 4.5.1]{GLS}.

\vs

\noindent \emph{Case 1. $n$ odd, $q \equiv 1 \imod{4}$}

\vs

Here $n = 2m+1 \geqs 5$ and either $q = 9$ or $q \in \{5, 17, 257, 65537, \ldots\}$ is a Fermat prime. In particular, $q-1$ is a $2$-power and $\a = 2(q-1)$. As recorded in Table \ref{tab:syl}, we have
\[
|H_0| = \a^m((q-1)_2)^m(m!)_2 = (q-1)^{n-1}((n-1)!)_2 = (q-1)^{n-1}(n!)_2
\]
and it follows that $H$ is contained in a subgroup $L$ of $G$ of type ${\rm GL}_1(q) \wr S_n$. The case $(n,q) = (5,9)$ was handled in Proposition \ref{p:comp}, so we may assume $n \geqs 7$ if $q=9$.

Suppose $x \in G$ is a $t_1$-involution. Then
\begin{equation}\label{e:psl1}
|x^G| = \frac{|{\rm GL}_n(q)|}{|{\rm GL}_{n-1}(q)||{\rm GL}_1(q)|} = \frac{q^{n-1}(q^n-1)}{q-1} = b_1
\end{equation}
and we claim that $H$ contains at most $a_1 = n+\binom{n}{2}(q-1)$ such elements. To justify the claim, we will count the number of involutions of type $(I_1,-I_{n-1})$ in the subgroup $J = {\rm GL}_1(q) \wr S_n$ of ${\rm GL}_n(q)$. Visibly, there are $\binom{n}{1} = n$ such elements in the base group $B = {\rm GL}_1(q)^n$ of $J$. And if $x \in B\pi$ with $1 \ne \pi \in S_n$, then $\pi$ is a transposition, $x$ is $B$-conjugate to $\pi$ and we have $|C_B(x)| = |{\rm GL}_1(q)|^{n-1}$. Since $S_n$ contains $\binom{n}{2}$ transpositions, it follows that there are exactly $\binom{n}{2}(q-1)$ such involutions in $J \setminus B$, whence $a_1$ in $J$ in total.

By Proposition \ref{p:inv}, if $x \in G$ is a non-$t_1$ involution then
\begin{equation}\label{e:psl2}
|x^G| \geqs \frac{|{\rm GL}_n(q)|}{|{\rm GL}_{n-2}(q)||{\rm GL}_2(q)|} = \frac{q^{2n-4}(q^{n-1}-1)(q^n-1)}{(q^2-1)(q-1)} = b_2
\end{equation}
(minimal if $x$ is a $t_2$-involution, using the fact that $(n,q) \ne (5,9)$) and we note that 
\[
|H| = |H_0||G:G_0|_2 \leqs (q-1)^{n-1}((n-1)!)_22^{1+\b} < 2^{n+\b}(q-1)^{n-1} = a_2,
\]
since $((n-1)!)_2 < 2^{n-1}$ by Lemma \ref{l:factorial}.

Then by applying Lemmas \ref{l:prob} and \ref{l:est}, we deduce that
\begin{equation}\label{e:two}
Q(G,H) < a_1^2b_1^{-1} + a_2^2b_2^{-1}
\end{equation}
and it is routine to check that this upper bound yields $Q(G,H) < 1$ and thus $b(G,H) = 2$.
  
\vs

\noindent \emph{Case 2. $n$ odd, $q \equiv 3 \imod{4}$}

\vs

Here $n = 2m+1 \geqs 5$ and $q \in \{3, 7, 31, 127, 8191, \ldots\}$ is a Mersenne prime. Note that $q+1$ is a $2$-power and $\a = 2(q+1)$. By inspecting Table \ref{tab:syl},
\[
|H_0| = \a^m((q-1)_2)^m(m!)_2 = (2\a)^m(m!)_2
\]
and we deduce that $H$ is contained in a subgroup $L$ of $G$ of type $({\rm GL}_2(q) \wr S_m) \oplus {\rm GL}_1(q)$. Here $L$ is the stabiliser in $G$ of a direct sum decomposition of the natural module of the form
\[
V = U \oplus W = (U_1 \oplus \cdots \oplus U_m) \oplus W,
\]
where $\dim W = 1$ and each $U_i$ is $2$-dimensional. The case $(n,q) = (5,3)$ was handled in Proposition \ref{p:comp}, so we may assume $q \geqs 7$ when $n = 5$. 

As in Case 1, if $x \in G$ is a $t_1$-involution then $|x^G| = b_1$ as in \eqref{e:psl1} and by counting the number of such elements in $L$ we deduce that there are at most 
\[
a_1 = 1 + \binom{m}{1}\frac{|{\rm GL}_2(q)|}{|{\rm GL}_1(q)|^2} = 1+mq(q+1)
\]
such involutions in $H$. And for any other involution $x \in G$ we see that \eqref{e:psl2} holds and we note that $|H| \leqs 2^{3m+1}(q+1)^m = a_2$. As before, the result now follows via the upper bound in \eqref{e:two}.

\vs

\noindent \emph{Case 3. $n$ even, $q \equiv 1 \imod{4}$}

\vs

Here $n = 2m \geqs 4$ and either $q=9$ or $q$ is a Fermat prime. From Table \ref{tab:syl} we read off
\[
|H_0| = \frac{1}{d}\a^m(q-1)^{m-1}(m!)_2 = \frac{1}{d}(q-1)^{n-1}(n!)_2
\]
and we see that $H$ is contained in a subgroup $L$ of type ${\rm GL}_1(q) \wr S_n$.

For $n \geqs 8$, we can now proceed as in Case 1 and we deduce that \eqref{e:two} holds, where $a_1,b_1$ and $b_2$ are defined as in Case 1, and we set $a_2 = 2^{n+1+\b}(q-1)^{n-1}$. It is easy to check that the bound in \eqref{e:two} implies that $Q(G,H) <1$ for all $n \geqs 8$ and $q \geqs 5$, so for the remainder of Case 3 we may assume $n = 4$ or $6$.

Suppose $n = 6$ and note that $d = (n,q-1) = 2$. In view of Proposition \ref{p:comp}, we may assume $q \geqs 17$. As above, the contribution to $\widehat{Q}(G,H)$ from $t_1$-involutions is at most $a_1^2b_1^{-1}$, where $a_1$ and $b_1$ are defined as before. As recorded in Proposition \ref{p:inv}, if $x \in G$ is any other involution then $|x^G|$ is minimal when $x$ is a $\gamma_1$-type  graph automorphism (rather than a $t_2$-involution) and thus
\[
|x^G| \geqs \frac{|G_0|}{|{\rm PSp}_6(q)|} = q^6(q^3-1)(q^5-1)=b_2.
\]
In addition, we have $|H| \leqs 2^{1+\b}(q-1)^5(6!)_2 = 2^{5+\b}(q-1)^5=a_2$. It is now routine to check that the upper bound in \eqref{e:two} is sufficient.

Finally, let us assume $n=4$. The groups with $q \in \{5,9,17\}$ were handled in Proposition \ref{p:comp}, so we may assume $q \geqs 257$. As usual, the contribution to $\widehat{Q}(G,H)$ from $t_1$-involutions is at most $a_1^2b_1^{-1}$. If $x \in G$ is a $\gamma_1$ graph automorphism, then 
\[
|x^G| \geqs \frac{|G_0|}{|{\rm PSp}_4(q)|} = \frac{1}{2}q^2(q^3-1)=b_2
\]
and the proof of \cite[Proposition 2.7]{FPR3} indicates that there are at most 
\[
a_2 = \frac{4!}{2!2^2}|{\rm GL}_1(q)| = 3(q-1)
\]
such elements in $L$. And since $q$ is a prime, we see that  
\[
|x^G| \geqs \frac{|{\rm GL}_4(q)|}{2|{\rm GL}_2(q^2)|} = \frac{1}{2}q^4(q-1)(q^3-1)=b_3
\]
for all other involutions $x \in G$. Now $|H| \leqs 2^{1+\b}(q-1)^3(4!)_2 = 2^{4+\b}(q-1)^3=a_3$, so 
\begin{equation}\label{e:three}
Q(G,H)< a_1^2b_1^{-1} + a_2^2b_2^{-1} + a_3^2b_3^{-1}
\end{equation}
and it is straightforward to check that this upper bound is less than $1$ for all $q \geqs 257$.

\vs

\noindent \emph{Case 4. $n$ even, $q \equiv 3 \imod{4}$}

\vs

Write $n = 2m \geqs 4$ and note that $q$ is a Mersenne prime. From Table \ref{tab:syl} we read off
\[
|H_0| = \frac{1}{2}\a^m(q-1)^{m-1}(m!)_2,
\]
which allows us to conclude that $H$ is contained in a subgroup $L$ of type ${\rm GL}_2(q) \wr S_m$. 

Let $x \in G$ be a $t_1$-involution and note that $|x^G| = b_1$, where $b_1$ is defined as in \eqref{e:psl1}. We calculate that there are at most 
\[
a_1 = \binom{m}{1}\frac{|{\rm GL}_2(q)|}{|{\rm GL}_1(q)|^2} = mq(q+1)
\]
such elements in $L$. If we now assume $n \geqs 8$, then any other involution $x \in G$ satisfies the bound $|x^G| \geqs b_2$ as in \eqref{e:psl2}. So for $n \geqs 8$ it follows that \eqref{e:two} holds with $a_2 = 2^{3m}(q+1)^m$ and we check that this bound yields $Q(G,H)<1$ unless $(n,q) = (8,3)$. But the latter case was handled in Proposition \ref{p:comp}, so the proof is complete for $n \geqs 8$.

Now assume $n = 6$. If $x \in G$ is a non-$t_1$ involution then $|x^G| \geqs q^6(q^3-1)(q^5-1) = b_2$ as in Case 3. If we now take $a_1$ and $b_1$ as above, noting that $|H| \leqs 2^7(q+1)^3 = a_2$ since $(3!)_2 = 2$, then the bound in \eqref{e:two} yields $Q(G,H)<1$ for all $q \geqs 7$ and we recall that the case $q = 3$ was treated in Proposition \ref{p:comp}. 

Finally, let us assume $n=4$. Define $a_1$ and $b_1$ as above and set $b_2 = \frac{1}{2}q^2(q^3-1)$ as in Case 3 (for $n=4$). Since $|H| \leqs 2^5(q+1)^2 = a_2$, it follows that \eqref{e:two} holds and one checks that this upper bound yields $Q(G,H) < 1$ for all $q > 127$ (recall that $q$ is a Mersenne prime, so the bound $q>127$ implies that $q \geqs 8191$). Finally, the result for $q \in \{3,7,31,127\}$ can be checked computationally (see Proposition \ref{p:comp}). 
\end{proof}

\begin{prop}\label{p:psu}
The conclusion to Theorem \ref{t:main2} holds if $G_0 = {\rm U}_n(q)$.
\end{prop}

\begin{proof}
As in the proof of the previous proposition, let $H$ be a Sylow $2$-subgroup of $G$ and set $H_0 = H \cap G_0$ and $\a = (q^2-1)_2$. Write $d = (n,q+1)$ and $\b = \delta_{9,q}$.

For an integer $1 \leqs k \leqs n/2$, we say that $x \in G$ is of type $t_k$ if it is conjugate in ${\rm PGU}_n(q)$ to $y = \hat{y}Z$, where $\hat{y} = (-I_k,I_{n-k}) \in {\rm GU}_n(q)$ is a diagonal matrix (with non-degenerate eigenspaces) and $Z$ is the centre of ${\rm GU}_n(q)$. As before, this notation is consistent with \cite[Table 4.5.1]{GLS}.

\vs

\noindent \emph{Case 1. $n$ odd, $q \equiv 1 \imod{4}$}

\vs

Write $n = 2m+1 \geqs 5$ and note that $\a = 2(q-1)$. By inspecting Table \ref{tab:syl}, we see that
\[
|H_0| = \a^m((q+1)_2)^m(m!)_2 = 2^{n-1}(q-1)^{m}(m!)_2
\]
and hence $H$ is contained in a subgroup $L$ of $G$ of type $({\rm GU}_2(q) \wr S_m)\perp {\rm GU}_1(q)$. Here the notation indicates that $L$ is the stabiliser in $G$ of an orthogonal decomposition  
\[
V = U \perp W = (U_1 \perp \cdots \perp U_m) \perp W
\]
of the natural module $V$ for $G_0$, where $W$ is a non-degenerate $1$-space and each $U_i$ is a non-degenerate $2$-space.

Suppose $x \in G$ is a $t_1$-involution. Then
\begin{equation}\label{e:psu1}
|x^G| = \frac{|{\rm GU}_n(q)|}{|{\rm GU}_{n-1}(q)||{\rm GU}_1(q)|} = \frac{q^{n-1}(q^n-(-1)^n)}{q+1} = b_1
\end{equation}
and by arguing as in the proof of Proposition \ref{p:psl} (see Case 2), we can show that there are at most 
\[
a_1 = 1+ \binom{m}{1}\frac{|{\rm GU}_2(q)|}{(q+1)^2} = 1+mq(q-1)
\]
such involutions in $L$ (and therefore at most $a_1$ in $H$ itself). And if $x \in G$ is any other type of involution, then Proposition \ref{p:inv} gives
\begin{equation}\label{e:psu2}
|x^G| \geqs \frac{|{\rm GU}_n(q)|}{|{\rm GU}_{n-2}(q)||{\rm GU}_2(q)|} = \frac{q^{2n-4}(q^{n-1}-(-1)^{n-1})(q^n-(-1)^n)}{(q^2-1)(q+1)} = b_2
\end{equation}
(minimal if $x$ is a $t_2$-involution). In addition, since 
$|G:G_0|_2 \leqs 2^{1+\b}$, we compute  
\[
|H| \leqs 2^{n+\b}(q-1)^{m}(m!)_2 < 2^{n+m+\b}(q-1)^{m} = a_2.
\]
Then the upper bound in \eqref{e:two} holds and it is easy to check that this gives $Q(G,H)< 1$ in all cases.

\vs

\noindent \emph{Case 2. $n$ odd, $q \equiv 3 \imod{4}$}

\vs

Here $n = 2m+1 \geqs 5$, $\a = 2(q+1)$ and we note that  
\[
|H_0| = \a^m((q+1)_2)^m(m!)_2 = 2^{m}(q+1)^{n-1}(m!)_2 = (q+1)^{n-1}(n!)_2
\]
(see Table \ref{tab:syl}). In particular, we deduce that $H$ is contained in a subgroup $L$ of $G$ of type ${\rm GU}_1(q) \wr S_n$.  

If $x \in G$ is a $t_1$-involution then \eqref{e:psu1} holds and by arguing as in the proof of Proposition \ref{p:psl} (see Case 1), we deduce that there are at most $a_1 = n + \binom{n}{2}(q+1)$ such involutions in $H$. And if $x \in G$ is any other type of involution, then \eqref{e:psu2} holds and we note that $|H| \leqs 2^n(q+1)^{n-1} = a_2$. By defining $b_1$ and $b_2$ as in Case 1, we deduce that \eqref{e:two} holds and this gives $Q(G,H)<1$ unless $(n,q) = (5,7)$, or if $q = 3$ and $n \leqs 25$.

So to complete the proof in Case 2, we may assume $(n,q) = (5,7)$, or $q = 3$ and $5 \leqs n \leqs 25$. Firstly, by setting $a_2 = 2(q+1)^{n-1}(n!)_2$ we find that the upper bound in \eqref{e:two} is sufficient unless $q = 3$ and $n \leqs 17$. In addition, we recall that the case $(n,q) = (5,3)$ was handled in Proposition \ref{p:comp}, so we may assume $q = 3$ and $7 \leqs n \leqs 17$ is odd. To deal with the remaining cases, define $a_1$, $b_1$ and $b_2$ as above, and observe that there are no more than
\begin{equation}\label{e:psua2}
\begin{aligned}
a_2 & = \binom{n}{2} + \binom{n}{2}|{\rm GU}_1(q)|\cdot \binom{n-2}{1} + \frac{n!}{2!(n-4)!2^2}|{\rm GU}_1(q)|^2 \\
& = \frac{1}{2}n(n-1)\left(1+(n-2)(q+1)+\frac{1}{4}(n-2)(n-3)(q+1)^2\right)
\end{aligned}
\end{equation}
involutions of type $t_2$ in $L$. By inspecting \cite[Table 4.5.1]{GLS}, we note that if $x \in G$ is any other involution (that is to say, $x$ is not a $t_k$-involution with $k=1,2$) then 
\begin{equation}\label{e:psu3}
|x^G| \geqs \frac{|{\rm GU}_n(q)|}{|{\rm GU}_{n-3}(q)||{\rm GU}_3(q)|} > \frac{1}{2}\left(\frac{q}{q+1}\right)q^{6n-18} = b_3.
\end{equation}
It follows that \eqref{e:three} holds with $a_3 = 2(q+1)^{n-1}(n!)_2$ and one checks that this upper bound is sufficient for all $n \geqs 7$.

\vs

\noindent \emph{Case 3. $n$ even, $q \equiv 1 \imod{4}$}

\vs

Write $n = 2m \geqs 4$ and note that $d = (n,q+1) = 2$. Then $\a = 2(q-1)$ and from Table \ref{tab:syl} we get 
\[
|H_0| = \frac{1}{2}\a^m((q+1)_2)^{m-1}(m!)_2,
\]
which means that $H$ is contained in a subgroup $L$ of $G$ of type ${\rm GU}_2(q) \wr S_m$. Note that if $x \in G$ is a $t_1$-involution then \eqref{e:psu1} holds and by arguing as in the proof of Proposition \ref{p:psl} (see Case 4), we deduce that there are at most 
\[
a_1 = \binom{m}{1}\frac{|{\rm GU}_2(q)|}{|{\rm GU}_1(q)|^2} = mq(q-1)
\]
such elements in $H$.

For now, let us assume $n \geqs 8$. Then Proposition \ref{p:inv} implies that $|x^G| \geqs b_2$ if $x$ is any other involution in $G$, where $b_2$ is defined in \eqref{e:psu2}. And since $|H| \leqs 2^{3m+\b}(q-1)^m = a_2$, we deduce that \eqref{e:two} holds, which in turn allows us to conclude that $Q(G,H)<1$. 

Now suppose $n = 6$. If $x \in G$ is a non-$t_1$ involution, then $|x^G| \geqs b_2$, where 
\begin{equation}\label{e:psu4}
b_2 = \frac{|G_0|}{|{\rm Sp}_6(q)|} = q^6(q^3+1)(q^5+1)
\end{equation}
(with equality possible if $x$ is a $\gamma_1$ graph automorphism). It follows that  
\eqref{e:two} holds with $a_2 = 2^{2+\b}|H_0| = 2^{7+\b}(q-1)^3$ and one checks that this upper bound yields $Q(G,H)<1$.

Finally, let us assume $n=4$. The cases $q \in \{5,9,17\}$ were handled in Proposition \ref{p:comp}, so we may assume $q \geqs 257$ is a Fermat prime. Here $|H| \leqs 2^5(q-1)^2 = a_2$ and we note that
\[
|x^G| \geqs \frac{|G_0|}{|{\rm Sp}_4(q)|} = \frac{1}{2}q^2(q^3+1) = b_2
\]
for every non-$t_1$ involution $x \in G$ (see Proposition \ref{p:inv}). It follows that \eqref{e:two} holds and we get  $Q(G,H)<1$ for all $q > 257$ (recall that $q$ is a Fermat prime, so the bound $q>257$ implies that $q \geqs 65537$). 

So to complete the argument, suppose $G_0 = {\rm U}_4(q)$ with $q=257$. Without loss of generality, we may assume $G = {\rm Aut}(G_0) = G_0.2^2$. As before, the contribution to $\widehat{Q}(G,H)$ from $t_1$-involutions is at most $a_1^2b_1^{-1}$, where $a_1$ and $b_1$ are defined as above. Also observe that if $x \in G$ is an involution, which is not of type $t_1$ nor $\gamma_1$, then 
\begin{equation}\label{e:psu5}
|x^G| \geqs \frac{|{\rm GU}_4(q)|}{2|{\rm GU}_2(q)|^2} = \frac{1}{2}q^4(q^2-q+1)(q^2+1) = f(q)
\end{equation}
(minimal if $x$ is a $t_2$-involution). Set $G_1 = {\rm PGU}_4(q) = G_0.2$ and $H_1 = H \cap G_1$, so we have $|H:H_1| = 2$. If $x \in H$ is a $\gamma_1$ graph automorphism, then $|x^G| = q^2(q^3+1) = b_2$ and we note that every graph automorphism in $H$ is contained in the coset $H_1x$, so there are at most $|H_1| = 2^{20} = a_2$ such elements in $H$. Then by setting $a_3 = |H| = 2^{21}$ and $b_3 = f(q)$, we deduce that \eqref{e:three} holds and this is good enough to force $Q(G,H)<1$.

\vs

\noindent \emph{Case 4. $n$ even, $q \equiv 3 \imod{4}$}

\vs

Here we have $n = 2m \geqs 4$, $\a = 2(q+1)$ and by inspecting Table \ref{tab:syl} we get 
\[
|H_0| = \frac{1}{d}\a^m(q+1)^{m-1}(m!)_2 = \frac{1}{d}(q+1)^{n-1}(n!)_2.
\]
Then $H$ is contained in a subgroup $L$ of $G$ of type ${\rm GU}_1(q) \wr S_n$. As usual, if $x \in G$ is a $t_1$-involution then $|x^G| = b_1$, where $b_1$ is given in \eqref{e:psu1}, and it is easy to check that there are at most $a_1 = n+\binom{n}{2}(q+1)$ such elements in $L$ (and therefore in $H$ as well).

Suppose $n \geqs 10$. If $x \in G$ is a $t_2$-involution then $|x^G| =b_2$, where $b_2$ is defined in \eqref{e:psu2}, and we see that there are at most $a_2$ such elements in $H$, where $a_2$ is defined in \eqref{e:psua2}. And if $x \in G$ is any other involution (that is, if $x$ is not a $t_k$-involution for $k=1$ or $2$) then $|x^G|$ is minimal when $x$ is a $t_3$-involution and thus $|x^G| > b_3$, where $b_3$ is defined in \eqref{e:psu3}. Then by setting $a_3 = 2^{n+1}(q+1)^{n-1}$ we deduce that \eqref{e:three} holds and the result follows.

The case $n=8$ is entirely similar. The only difference is that if $x \in G$ is an involution, which is not of type $t_1$ or $t_2$, then $|x^G|$ is minimal when $x$ is a $\gamma_1$ graph automorphism. This means that  
\[
|x^G| \geqs \frac{|G_0|}{|{\rm Sp}_8(q)|} \geqs \frac{1}{8}q^{12}(q^3+1)(q^5+1)(q^7+1) = b_3 
\]
and one checks that the upper bound in \eqref{e:three} is sufficient for all $q \geqs 7$. This completes the argument because the case $G_0 = {\rm U}_8(3)$ was handled in Proposition \ref{p:comp}.

Next assume $n = 6$. The case $q=3$ is treated in Proposition \ref{p:comp}, so we may assume $q \geqs 7$. As usual, the contribution to $\widehat{Q}(G,H)$ from $t_1$-involutions is at most $a_1^2b_1^{-1}$, where $a_1$ and $b_1$ are defined as above. And if $x \in G$ is any other involution, then $|x^G| \geqs b_2$, where $b_2$ is defined in \eqref{e:psu4}. Since $|H| \leqs 2^5(q+1)^5 = a_2$, we deduce that \eqref{e:two} holds and the result follows for all $q \geqs 31$. 

So to complete the proof for $G_0 = {\rm U}_6(q)$, we may assume $q = 7$. If $x \in G$ is a $\gamma_1$ involution then $|x^G| = b_2$ as in \eqref{e:psu4} and by appealing to the proof of \cite[Proposition 2.7]{FPR3} we deduce that there are at most 
\[
a_2 = \frac{6!}{3!2^3} \cdot |{\rm GU}_1(q)|^2 = 15(q+1)^2
\]
such elements in $L$ (and hence at most this number in $H$). And if $x \in G$ is an involution, which is not of type $t_1$ or $\gamma_1$, then $|x^G|$ is minimal when $x$ is of type $t_2$ and thus 
\[
|x^G| \geqs \frac{|{\rm GU}_6(q)|}{|{\rm GU}_4(q)||{\rm GU}_2(q)|} = \frac{q^8(q^5+1)(q^6-1)}{(q^2-1)(q+1)}= b_3.
\]
Since $|H_0| = 2^{18}$ it follows that $|H| \leqs 2^{20} = a_3$ and the bound in \eqref{e:three} is satisfied. This forces $Q(G,H)<1$ and the result follows.

Finally, let us assume $n=4$ and define $a_1$ and $b_1$ as above. If $x \in G$ is a $\gamma_1$ involution, then $|x^G| \geqs \frac{1}{2}q^2(q^3+1) = b_2$ and by inspecting the proof of \cite[Proposition 2.7]{FPR3} we calculate that there are no more than $a_2 = 3(q+1)$ such elements in $H$. And if $x \in G$ is any other involution (neither a $t_1$ nor a $\gamma_1$ involution), then $|x^G| \geqs b_3 = f(q)$ as defined in \eqref{e:psu5} and we note that $|H| \leqs 2^4(q+1)^3 = a_3$. Then \eqref{e:three} holds and the reader can check that this bound gives $Q(G,H)<1$ for all $q \geqs 31$. This completes the proof since the cases $q=3,7$ were handled in Proposition \ref{p:comp}. 
\end{proof}

\begin{prop}\label{p:psp}
The conclusion to Theorem \ref{t:main2} holds if $G_0 = {\rm PSp}_n(q)$.
\end{prop}

\begin{proof}
Here $n = 2m \geqs 4$ and we set $H_0 = H \cap G_0$ and $\a = (q^2-1)_2$ as before, where $H$ is a Sylow $2$-subgroup of $G$. As in \cite[Table 4.5.1]{GLS}, we say that an involution $x \in G$ is of type $t_1$ if it is $G_0$-conjugate to $y = \hat{y}Z$, where $\hat{y} = (-I_2,I_{n-2}) \in {\rm Sp}_n(q)$ is a diagonal matrix (with non-degenerate eigenspaces) and $Z$ is the centre of ${\rm Sp}_n(q)$. As usual, set $\b = \delta_{9,q}$.

\vs

\noindent \emph{Case 1. $n \geqs 6$}

\vs

By inspecting Table \ref{tab:syl}, we see that $|H_0| = \frac{1}{2}\a^m(m!)_2$ and thus $H$ is contained in a subgroup $L$ of $G$ of type ${\rm Sp}_2(q) \wr S_m$. In particular, if $x$ is a $t_1$-involution, then 
\begin{equation*}
|x^G| = \frac{|{\rm Sp}_{n}(q)|}{|{\rm Sp}_{n-2}(q)||{\rm Sp}_2(q)|} = \frac{q^{n-2}(q^n-1)}{q^2-1} = b_1
\end{equation*}
and by counting in $L$ we deduce that there are at most 
\begin{equation*}
a_1 = \binom{m}{1} + \binom{m}{2}|{\rm Sp}_2(q)| = m\left(1+\frac{1}{2}(m-1)q(q^2-1)\right)
\end{equation*}
such elements in $H$.

First assume $n \geqs 10$. If $x \in G$ is a non-$t_1$ involution, then Proposition \ref{p:inv} implies that 
\begin{equation*}
|x^G| \geqs \frac{|{\rm Sp}_{n}(q)|}{|{\rm Sp}_{n-4}(q)||{\rm Sp}_4(q)|} > \frac{1}{2}q^{4n-16} = b_2.
\end{equation*}
In addition, $|H| \leqs 2^{n+\b}(q+1)^{m} = a_2$ and one can check that the upper bound in \eqref{e:two} is good enough unless $(n,q) = (10,3)$. But the latter case was handled in Proposition \ref{p:comp}, so the result follows.

Similarly, if $n=8$ and $x$ is a non-$t_1$ involution, then 
\[
|x^G| \geqs \frac{|{\rm Sp}_{8}(q)|}{2|{\rm Sp}_{4}(q^2)|}  = \frac{1}{2}q^8(q^2-1)(q^6-1) = b_2
\]
and we have $|H| \leqs 2^{7+\b}(q+1)^4 = a_2$. Then \eqref{e:two} holds and we immediately deduce that $Q(G,H) < 1$ if $q \geqs 5$. This completes the argument since we treated the case $G_0 = {\rm PSp}_8(3)$ in Proposition \ref{p:comp}.

Next assume $n=6$. If $x$ is a non-$t_1$ involution, then 
\[
|x^G| \geqs \frac{|{\rm Sp}_{6}(q)|}{2|{\rm GU}_{3}(q)|} = \frac{1}{2}q^{6}(q-1)(q^2+1)(q^3-1) = b_2
\]
(see Proposition \ref{p:inv}) and we note that $|H| \leqs 2^{4+\b}(q+1)^3 = a_2$. Then \eqref{e:two} holds and we deduce that $Q(G,H)<1$ for all $q \geqs 5$. This just leaves the case $G_0 = {\rm PSp}_6(3)$, which was handled in Proposition \ref{p:comp}.

\vs

\noindent \emph{Case 2. $n=4$}

\vs

For the remainder of the proof, we may assume $n = 4$. The groups with $q \leqs 17$ can be handled using {\sc Magma} (see Proposition \ref{p:comp}), so we may assume $q \geqs 31$ is a Fermat or Mersenne prime. In particular, $G = {\rm PGSp}_4(q) = G_0.2$. Write $G_0 = L/Z$ and $H_0 = J/Z$, where $Z$ is the centre of $L = {\rm Sp}_4(q)$, and note that
\begin{align*}
J & = Q_{2(q-\e)} \wr S_2 < {\rm Sp}_2(q) \wr S_2 \\
H & = D_{2(q-\e)} \wr S_2,
\end{align*}
where $\e = 1$ if $q \equiv 1\imod{4}$, otherwise $\e= -1$ (here $Q_{2(q-\e)}$ denotes the quaternion group of order $2(q-\e)$). In particular, $|H| = 8(q-\e)^2$. Given a finite group $X$, we recall that $i_2(X)$ denotes the number of involutions in $X$.

To begin with, we will assume $q \equiv 1 \imod{4}$, so $q \geqs 257$ is a Fermat prime and
\[
i_2(H) = (q+1)^2 -1 + 2(q-1) = q^2+4q-2.
\]
There are four $G$-classes of involutions, with representatives labelled $t_1,t_2,t_1'$ and $t_2'$ in \cite[Table 4.5.1]{GLS}. Here $t_1,t_2 \in G_0$ and $t_1',t_2' \in G \setminus G_0$. It will be useful to note that the involutions of type $t_1$ lift to involutions in $L$, whereas those of type $t_2$ lift to elements of order $4$.

Suppose $x \in G$ is a $t_1$-involution. Then 
\[
|x^G| = \frac{|{\rm Sp}_{4}(q)|}{2|{\rm Sp}_2(q)|^2} = \frac{1}{2}q^2(q^2+1) = b_1
\]
and we note that $x = \hat{x}Z \in L/Z$ with $\hat{x} \in L$ an involution of type $(-I_2,I_2)$. Now 
\[
i_2(J) = 2q+|Z|-1 = 2q+1
\]
since $Q_{2(q-1)}$ has a unique involution, and we deduce that exactly $q$ involutions in $H_0$ lift to an involution in $J$. This means that $H$ contains $a_1 = q$ involutions of type $t_1$. Similarly, we calculate that there are precisely $q^2+4q-1$ elements $y \in J$ of order $4$ with $y^2 \in Z$ and this allows us to conclude that $H$ contains $(q^2+4q-1)/2$ involutions of type $t_2$.

Now let $x \in G$ be an involution of type $t_1'$ in the notation of \cite[Table 4.5.1]{GLS}. Here $x \in G \setminus G_0$ and
\[
|x^G| = \frac{|{\rm Sp}_{4}(q)|}{2|{\rm Sp}_2(q^2)|} = \frac{1}{2}q^2(q^2-1) = b_2.
\]
Moreover, our previous calculations show that there are at most 
\[
a_2 = i_2(H) - q - \frac{1}{2}(q^2+4q-1) = \frac{1}{2}(q^2+2q-3)
\]
such involutions in $H$.

Now, if $x \in G$ is an involution of type $t_2$ or $t_2'$, then 
\[
|x^G| \geqs \frac{|{\rm Sp}_{4}(q)|}{2|{\rm GU}_{2}(q)|} = \frac{1}{2}q^{3}(q-1)(q^2+1) = b_3
\]
and we note that $|H| = 8(q-1)^2 = a_3$. It follows that \eqref{e:three} holds and we conclude that $Q(G,H)<1$ for all $q \geqs 257$.

A very similar argument applies when $q \equiv 3 \imod{4}$. Here the involutions in $G$ of type $t_1$ and $t_2'$ are contained in $G_0$, whereas those of type $t_1'$ and $t_2$ are in $G \setminus G_0$. Since $H = D_{2(q+1)} \wr S_2$ we compute
\[
i_2(H) = (q+3)^2 - 1 +2(q+1) = q^2+8q+10
\]
and as before we calculate that $H$ contains $a_1 = q+2$ involutions of type $t_1$. By counting the elements $y \in J$ of order $4$ with $y^2 \in Z$, we deduce that there are $(q^2+8q+11)/2$ involutions of type $t_2'$ in $H$. As a consequence, it follows that $H$ contains at most 
\[
a_2 = i_2(H) - (q+2) - \frac{1}{2}(q^2+8q+11) = \frac{1}{2}(q^2+6q+5)
\]
involutions of type $t_1'$. Putting this together, we deduce that \eqref{e:three} holds with $a_3 = 8(q+1)^2$ and $a_1,a_2,b_1,b_2,b_3$ defined as above. It is now easy to check that this upper bound on $Q(G,H)$ is less than $1$ for all $q \geqs 31$.
\end{proof}

\begin{prop}\label{p:o_odd}
The conclusion to Theorem \ref{t:main2} holds if $G_0 = \O_n(q)$.
\end{prop}

\begin{proof}
Here $n = 2m+1 \geqs 7$ and we set $H_0 = H \cap G_0$, where $H$ is a Sylow $2$-subgroup of $G$. For $k \in \{1,2\}$, we will say that an involution $x \in {\rm SO}_n(q) = G_0.2$ is of type $s_k$ if it is conjugate to an element of the form $(-I_{n-1},I_1)$ for $k=1$, or $(-I_2,I_{n-2})$ for $k = 2$ (in the notation of \cite[Table 4.5.1]{GLS}, these are the involutions of type $t_m$ or $t_m'$ (for $k=1$), and type $t_1$ or $t_1'$ (for $k=2$)). Note that if $x \in G$ is an involution of type $s_1$ then
\begin{equation}\label{e:o_odd1}
|x^G| \geqs \frac{|{\rm SO}_n(q)|}{2|{\rm SO}_{n-1}^{-}(q)|} = \frac{1}{2}q^{m}(q^{m}-1) = b_1.
\end{equation}
And similarly, if $x$ is an $s_2$-type involution then
\begin{equation}\label{e:o_odd2}
|x^G| \geqs \frac{|{\rm SO}_n(q)|}{2|{\rm SO}_{n-2}(q)||{\rm SO}_{2}^{-}(q)|} = \frac{q^{n-2}(q^{n-1}-1)}{2(q+1)} = b_2.
\end{equation}
As recorded in Table \ref{tab:syl}, we have $|H_0| = \frac{1}{2}\a^m(m!)_2$, where $\a = (q^2-1)_2$. Let $V$ be the natural module for $G_0$. As before, set $\b = \delta_{9,q}$.

\vs

\noindent \emph{Case 1. $q \equiv 1 \imod{4}$}

\vs

First we assume $q \equiv 1 \imod{4}$, so $|H_0| = \frac{1}{2}(2(q-1))^m(m!)_2$ and we note that $H$ is contained in a subgroup $L$ of $G$ of type 
\[
({\rm O}_2^{+}(q) \wr S_m) \perp {\rm O}_{1}(q) < {\rm O}_{2m}^{+}(q) \perp {\rm O}_1(q).
\]
Note that $L$ is the stabiliser in $G$ of an orthogonal decomposition of the form
\[
V = U \perp W = (U_1 \perp U_2 \perp \cdots \perp U_m) \perp W,
\]
where each $U_i$ is a $2$-dimensional non-degenerate orthogonal space of plus-type and $W$ is a non-degenerate $1$-space (recall that a $2$-dimensional orthogonal space is of \emph{plus-type} if it contains a non-zero singular vector with respect to the defining quadratic form). In particular, if $G = {\rm SO}_n(q) = G_0.2$, then
\[
L = D_{2(q-1)} \wr S_m = {\rm O}_2^{+}(q) \wr S_m < {\rm O}_{2m}^{+}(q) < G.
\]

Let $x \in G$ be an $s_1$-type involution, so $|x^G| \geqs b_1$ as in \eqref{e:o_odd1}. Working in $L$, and noting that $D_{2(q-1)}$ has exactly $q-1$ non-central involutions, we calculate that $H$ contains at most 
\[
a_1 = \binom{m}{1}(q-1) +1 = m(q-1)+1
\]
such elements. And if $x \in G$ is any other type of involution, then Proposition \ref{p:inv} implies that $|x^G| \geqs b_2$ with $b_2$ defined as in \eqref{e:o_odd2} and we note that $|H| \leqs 2^{n-1+\b}(q-1)^m = a_2$. It follows that $Q(G,H)$ satisfies the upper bound in \eqref{e:two} and this allows us to deduce that $Q(G,H)<1$ unless $(n,q) \in \{ (7,5), (7,9), (9,5)\}$. In the latter cases, by setting $a_2 = 2^{m+\b}(q-1)^m(m!)_2$, one can check that \eqref{e:two} is sufficient. 

\vs

\noindent \emph{Case 2. $q \equiv 3 \imod{4}$}

\vs

Here $|H_0| = \frac{1}{2}(2(q+1))^m(m!)_2$ and $H$ is contained in a subgroup $L$ of $G$ of type 
\[
({\rm O}_2^{-}(q) \wr S_m) \perp {\rm O}_{1}(q) < {\rm O}_{2m}^{\e}(q) \perp {\rm O}_1(q),
\]
where $\e = +$ if $m$ is even, otherwise $\e=-$. By arguing as in Case 1, we deduce that the contribution to $\widehat{Q}(G,H)$ from $s_1$-involutions is at most $a_1^2b_1^{-1}$, where $a_1 = m(q+1)+1$ and $b_1$ is defined in \eqref{e:o_odd1}. And if $x \in G$ is an $s_2$-type involution, then $|x^G| \geqs b_2$ as above and we note that $|H| \leqs 2^{n-1}(q+1)^m = a_2$. Then as in Case 1, we deduce that \eqref{e:two} holds and for $q \geqs 7$ we check that this bound is sufficient unless $q = 7$ and $n \in \{7,9,11\}$. But in each of these cases, by setting $a_2 = 2^m(q+1)^m(m!)_2$, we obtain $Q(G,H)<1$ via \eqref{e:two}.

Now suppose $q=3$. In view of Proposition \ref{p:comp}, we may assume $n \geqs 15$. Here we need a more accurate estimate for the number of $s_2$-involutions in $H$. By working in $L$, we calculate that there are at most $a_2$ such elements, where
\begin{align*}
a_2 & = \binom{m}{1}+\binom{m}{1}(q+1) + \binom{m}{2}(q+1)^2 + \binom{m}{2}|{\rm O}_2^{-}(q)| \\
& = m(q+2)+\frac{1}{2}m(m-1)\left((q+1)^2+2(q+1)\right).
\end{align*}
And if $x \in G$ is an involution, which is not of type $s_1$ or $s_2$, then it is easy to see that 
\[
|x^G| \geqs \frac{|{\rm SO}_n(q)|}{2|{\rm SO}_{n-3}^{-}(q)||{\rm SO}_3(q)|} > \frac{1}{4}q^{3n-9} = b_3.
\]
Setting $a_3 = 2^{n-1}(q+1)^m$, it follows that \eqref{e:three} holds and one can check that this upper bound is sufficient if $n \geqs 17$. Finally, if $n=15$ then $(m!)_2 = 2^4$ and so we can set $a_3 = 2^{11}(q+1)^7$ in \eqref{e:three} and we deduce that $Q(G,H)<1$ as required.
\end{proof}

\begin{prop}\label{p:o_even1}
The conclusion to Theorem \ref{t:main2} holds if $G_0 = {\rm P\O}_n^{+}(q)$.
\end{prop}

\begin{proof}
Here $n = 2m \geqs 8$ and as usual we set $H_0 = H \cap G_0$, where $H$ is a Sylow $2$-subgroup of $G$. Set $d = (4,q^m-1)$ and $\b = \delta_{9,q}$.

For $k \in \{1,2\}$, we say that an involution $x \in {\rm PO}_n^{+}(q)$ is of type $s_k$ if it lifts to an involution in ${\rm O}_n^{+}(q)$ of the form $(-I_{n-k},I_k)$. If $x$ is an  $s_1$-type involution, then $x \in {\rm PO}_n^{+}(q) \setminus {\rm PSO}_n^{+}(q)$ is an involutory graph automorphism of $G_0$ (of type $\gamma_1$ in \cite[Table 4.5.1]{GLS}) and we have
\begin{equation}\label{e:oeven1}
|x^G| \geqs \frac{|{\rm SO}_n^{+}(q)|}{2|{\rm SO}_{n-1}(q)|} = \frac{1}{2}q^{m-1}(q^m-1) = b_1.
\end{equation}
Similarly, if $x$ is an $s_2$-involution then
\begin{equation}\label{e:oeven2}
|x^G| \geqs \frac{|{\rm SO}_n^{+}(q)|}{2|{\rm SO}_{n-2}^{-}(q)||{\rm SO}_{2}^{-}(q)|} = \frac{q^{n-2}(q^m-1)(q^{m-1}-1)}{2(q+1)} = b_2.
\end{equation}
Recall that ${\rm Inndiag}(G_0)$ is the index-two subgroup of ${\rm PGO}_{n}^{+}(q)$ generated by the inner and diagonal automorphisms of $G_0$.  

\vs

\noindent \emph{Case 1. $q \equiv 1 \imod{4}$}

\vs

Here $d = 4$ and we first observe that
\[
|H_0| = \frac{1}{8}(2(q-1))^m(m!)_2,
\]
which means that $H$ is contained in a subgroup $L$ of $G$ of type ${\rm O}_{2}^{+}(q) \wr S_m$.
 
To begin with, let us assume $n \geqs 10$. If $x \in G$ is an $s_1$-type involution, then $|x^G| \geqs b_1$ as in \eqref{e:oeven1} and by counting in $L$ we deduce that there are at most $a_1 = m(q-1)$ such elements in $H$. Similarly, if $x$ is an $s_2$-type involution, then $|x^G| \geqs b_2$ (see \eqref{e:oeven2}) and we calculate that $H$ contains no more than $a_2$ such involutions, where
\[
a_2 = \binom{m}{1} + \binom{m}{2}(q-1)^2 + \binom{m}{2}|{\rm O}_{2}^{+}(q)| = m+\frac{1}{2}m(m-1)((q-1)^2+2(q-1)).
\]
And if $x \in G$ is any other type of involution, then by inspecting \cite[Table 4.5.1]{GLS} we deduce that $|x^G|$ is minimal when $x$ is an involutory graph automorphism of type $(-I_3,I_{n-3})$. So in this situation we get 
\begin{equation}\label{e:oeven3}
|x^G| \geqs \frac{|{\rm SO}_n^{+}(q)|}{2|{\rm SO}_{n-3}(q)||{\rm SO}_{3}(q)|} > \frac{1}{4}q^{3n-9} = b_3
\end{equation}
and we note that $|H| \leqs 2^{n+\b}(q-1)^m = a_3$. Putting all of this together, we deduce that \eqref{e:three} holds and this implies that $Q(G,H)<1$. The result follows.

Now assume $n = 8$ and $q \equiv 1 \imod{4}$. To begin with, we will assume $q \geqs 5$ is a Fermat prime, noting that a separate argument for $q = 9$ will be given below.

As above, the contribution to $\widehat{Q}(G,H)$ from $s_1$-involutions is at most $a_1^2b_1^{-1}$. Next suppose $x \in G$ is an $s_2$-involution and recall that $|x^G| \geqs b_2$. If $C_G(x)$ is of type ${\rm O}_6^{\e}(q) \times {\rm O}_2^{\e}(q)$, then $x$ is conjugate under a triality graph automorphism $\tau$ of $G_0$ to an involution $y$ with $C_G(y)$ of type ${\rm GL}_4^{\e}(q)$. And in turn, $z = y^{\tau}$ is contained in a second ${\rm Inndiag}(G_0)$-class of involutions with $C_G(z)$ of type ${\rm GL}_4^{\e}(q)$. Since the number of $s_2$-involutions in $H$ is at most $4+6(q-1)^2+12(q-1)$, as noted above, it follows that there are no more than
\[
a_2 = 3(4+6(q-1)^2+12(q-1))
\]
involutions $x \in H$ with $C_G(x)$ of type ${\rm O}_6^{\e}(q) \times {\rm O}_2^{\e}(q)$ or ${\rm GL}_4^{\e}(q)$ for $\e = \pm$. And if $x$ is any other involution, then the fact that $q$ is a prime (so $G$ does not contain any involutory field or graph-field automorphisms) implies that
\begin{equation}\label{e:o8}
|x^G| \geqs \frac{|{\rm SO}_8^{+}(q)|}{2|{\rm SO}_{5}(q)||{\rm SO}_{3}(q)|} = \frac{1}{2}q^7(q^2+1)(q^6-1) = b_3
\end{equation}
as in \eqref{e:oeven3}. Finally, since $|H| \leqs 2^7(q-1)^4 = a_3$ we deduce that \eqref{e:three} holds and this upper bound yields $Q(G,H)<1$, as required.

To complete the proof for $n = 8$ with $q \equiv 1 \imod{4}$, we may assume $q = 9$. The analysis of this case is complicated by the fact that $G$ may contain involutory field or graph-field automorphisms. Using {\sc Magma} \cite{magma}, we begin by constructing $L = {\rm Aut}(G_0) = G_0.[48]$ as a permutation group of degree $1795800$ and we construct a Sylow $2$-subgroup $J$ of $L$. We then take a set of representatives of the conjugacy classes of involutions in $J$ and we deduce that $i_2(J) = 14495$. In addition, by computing $|x^L|$ for each involution $x \in J$, we find that $J$ contains exactly $a_1 = 32$ involutions of type $s_1$. As before, if $x$ is any other involution in $G$, then $|x^G| \geqs b_2$ as in \eqref{e:oeven2}. So by applying Lemma \ref{l:est}, we see that \eqref{e:two} holds with $b_1$ defined in \eqref{e:oeven1} and $a_2 = 14463$. It is easy to check that this estimate gives  $Q(G,H)<1$ and the result follows.

\vs

\noindent \emph{Case 2. $q \equiv 3 \imod{4}$}

\vs

First assume $m$ is even and note that
\[
|H_0| = \frac{1}{2d}(2(q+1))^m(m!)_2.
\]
It follows that $H$ is contained in a subgroup $L$ of $G$ of type ${\rm O}_{2}^{-}(q) \wr S_m$, which means that the analysis of this case is almost identical to our previous work in Case 1. 

First assume $n \geqs 12$. Here we see that \eqref{e:three} holds, where
\[
a_1 = m(q+1), \; a_2 = \frac{1}{2}m(m-1)((q+1)^2+2(q+1))+m,\; a_3 = 2^n(q+1)^m
\]
and $b_1$, $b_2$ and $b_3$ are defined in \eqref{e:oeven1}, \eqref{e:oeven2} and \eqref{e:oeven3}, respectively. Then one can check that this upper bound is sufficient unless $q=3$ and $n \in \{12,16,20\}$. The case $(n,q) = (12,3)$ was handled in Proposition \ref{p:comp}. For $(n,q) = (20,3)$ we have $|H| \leqs 2^{18}4^{10}$ since $(10!)_2 = 2^8$ and by setting $a_3 = 2^{18}4^{10}$ we check that the revised bound in \eqref{e:three} is good enough. 

Finally, suppose $(n,q) = (16,3)$. We may assume $G = {\rm PGO}_{16}^{+}(3) = G_0.D_8$. Set $H_1 = H \cap {\rm Inndiag}(G_0)$ and note that $|H:H_1| = 2$. If $x \in H$ is an involutory graph automorphism of type $(-I_3,I_{13})$, then 
\[
|x^G| \geqs \frac{|{\rm SO}_{16}^{+}(3)|}{2|{\rm SO}_{13}(3)||{\rm SO}_3(3)|} = 2279214355760562960 = b_3
\]
and we note that all such involutions are contained in the coset $H_1x$, so there are at most $|H_1| = 2^{30} = a_3$ such elements in $H$. And if $x \in G$ is an involution that is not of type $s_1$ or $s_2$, nor a graph automorphism of type $(-I_3,I_{13})$, then 
\[
|x^G| \geqs \frac{|{\rm SO}_{16}^{+}(3)|}{2|{\rm SO}_{12}^{-}(3)||{\rm SO}_{4}^{-}(3)|} = 40320213639146662987584 = b_4.
\]
We deduce that 
\[
Q(G,H) < a_1^2b_1^{-1} + a_2^2b_2^{-1} + a_3^2b_3^{-1} + a_4^2b_4^{-1} < 1,
\]
where $a_4 = |H| = 2^{31}$, and the result follows.

Now assume $n=8$. In view of Proposition \ref{p:comp}, we are free to assume that $q \geqs 7$. Then by arguing as in Case 1, we deduce that  \eqref{e:three} holds, where $b_1$ and $b_2$ are defined as above, $b_3$ is given in \eqref{e:o8}, and we set
\[
a_1 = 4(q+1), \; a_2 = 18((q+1)^2+2(q+1))+12, \; a_3 = 2^7(q+1)^4,
\]
noting that $|H| \leqs a_3$. It is easy to check that this yields $Q(G,H)<1$ for all $q \geqs 7$ and the result follows.

To complete the proof, we may assume $m \geqs 5$ is odd. Here
\[
|H_0| = (2(q+1))^{m-1}((m-1)!)_2
\]
and we deduce that $H$ is contained in a subgroup $L$ of $G$ of type 
\[
({\rm O}_{2}^{-}(q) \wr S_{m-1}) \perp {\rm O}_{2}^{+}(q) < {\rm O}_{n-2}^{+}(q) \perp {\rm O}_{2}^{+}(q).
\]
If $x \in G$ is an $s_1$-involution, then $|x^G| \geqs b_1$ as in \eqref{e:oeven1} and by working in $L$ we deduce that there are at most 
\[
a_1 = \binom{m-1}{1}(q+1) + q-1 = m(q+1)-2
\]
such elements in $H$. Similarly, if $x$ is an $s_2$-involution then $|x^G| \geqs b_2$ with $b_2$ defined as in \eqref{e:oeven2} and we calculate that $H$ contains no more than $a_2$ such elements, where
\begin{align*}
a_2 & = \binom{m-1}{1} + \binom{m-1}{1}(q+1)(q-1) + \binom{m-1}{2}(q+1)^2  + \binom{m-1}{2}|{\rm O}_{2}^{-}(q)| + 1 \\
 & = \frac{1}{2}(m-1)(m-2)((q+1)^2+2(q+1)) + (m-1)q^2 + 1.
 \end{align*}
And if $x \in G$ is any other type of involution, then as in \eqref{e:oeven3} we have $|x^G|>\frac{1}{4}q^{3n-9} = b_3$ and we note that $|H| \leqs 2^n(q+1)^{m-1} = a_3$. Then one checks that the upper bound in \eqref{e:three} is sufficient unless $q=3$ and $n \in \{10,14\}$. The case $(n,q) = (10,3)$ was handled in Proposition \ref{p:comp}. And for $(n,q) = (14,3)$ we have $|H| \leqs 2^{24}$; by setting $a_3 = 2^{24}$ we find that the bound in \eqref{e:three} is good enough and the proof is complete.
\end{proof}

\begin{prop}\label{p:o_even2}
The conclusion to Theorem \ref{t:main2} holds if $G_0 = {\rm P\O}_n^{-}(q)$.
\end{prop}

\begin{proof}
Here $n = 2m \geqs 8$ is even. Let $H$ be a Sylow $2$-subgroup of $G$ and set $H_0 = H \cap G_0$, $d = (4,q^m+1)$ and $\b = \delta_{9,q}$.

For $k \in \{1,2\}$, we define an involution $x \in {\rm PO}_n^{-}(q)$ to be of type $s_k$ if it lifts to an involution in ${\rm O}_n^{-}(q)$ of the form $(-I_{n-k},I_k)$. If $x$ is of type $s_1$, then 
\begin{equation*}
|x^G| \geqs \frac{|{\rm SO}_n^{-}(q)|}{2|{\rm SO}_{n-1}(q)|} = \frac{1}{2}q^{m-1}(q^m+1) = b_1
\end{equation*}
and we get
\begin{equation*}
|x^G| \geqs \frac{|{\rm SO}_n^{-}(q)|}{2|{\rm SO}_{n-2}^{+}(q)||{\rm SO}_{2}^{-}(q)|} = \frac{q^{n-2}(q^m+1)(q^{m-1}+1)}{2(q+1)} = b_2
\end{equation*}
if $x$ is of type $s_2$. As in the proof of the previous proposition, we divide the analysis into two cases, according to whether or not $q \equiv 1 \imod{4}$ or $q \equiv 3 \imod{4}$.

\vs

\noindent \emph{Case 1. $q \equiv 1 \imod{4}$}

\vs

Here $d=2$ and 
\[
|H_0| = (2(q-1))^{m-1}((m-1)!)_2,
\]
which means that $H<L<G$, where $L$ is of type 
\[
({\rm O}_{2}^{+}(q) \wr S_{m-1}) \perp {\rm O}_{2}^{-}(q) < {\rm O}_{n-2}^{+}(q) \perp {\rm O}_{2}^{-}(q).
\]
Then by arguing as in Case 2 (with $m$ odd) in the proof of Proposition  \ref{p:o_even1}, we calculate that $H$ contains at most $a_1 = m(q-1)+2$ involutions of type $s_1$ and at most 
\[
a_2 = \frac{1}{2}(m-1)(m-2)((q-1)^2+2(q-1)) + (m-1)q^2 + 1
\]
involutions of type $s_2$. And if $x \in G$ is any other type of involution, then $|x^G|>\frac{1}{4}q^{3n-9} = b_3$ as in \eqref{e:oeven3} and we note that $|H| \leqs 2^{n+\b}(q-1)^{m-1} = a_3$. It is now routine to check that the upper bound in \eqref{e:three} is sufficient.

\vs

\noindent \emph{Case 2. $q \equiv 3 \imod{4}$}

\vs

First assume $m$ is even. Here $|H_0| = (2(q+1))^{m-1}((m-1)!)_2$ and thus $H$ is contained in a subgroup of type 
\[
({\rm O}_{2}^{-}(q) \wr S_{m-1}) \perp {\rm O}_{2}^{+}(q) < {\rm O}_{n-2}^{-}(q) \perp {\rm O}_{2}^{+}(q).
\]
Then by arguing as in Case 1, setting $b_1,b_2$ and $b_3$ as above, we deduce that \eqref{e:three} holds with  
\[
a_1 = m(q+1)-2, \; a_2 = \frac{1}{2}(m-1)(m-2)((q+1)^2+2(q+1)) + (m-1)q^2 + 1
\]
and $a_3 = 2^{n-1}(q+1)^{m-1}$. It is easy to check that this bound is sufficient unless $q = 3$ and $n \in \{8,12\}$. The latter two cases were handled in Proposition \ref{p:comp} and the result follows.

Finally, suppose $m$ is odd.  In this case we have
\[
|H_0| = \frac{1}{8}(2(q+1))^m(m!)_2
\]
and thus $H<L<G$ with $L$ of type ${\rm O}_{2}^{-}(q) \wr S_m$. And so by arguing as in Case 2 (with $m$ even) in the proof of Proposition \ref{p:o_even1}, we deduce that \eqref{e:three} holds, where $b_1,b_2,b_3$ are defined as above and 
\[
a_1 = m(q+1), \; a_2 = \frac{1}{2}m(m-1)((q+1)^2+2(q+1))+m,\; a_3 = 2^{n-1}(q+1)^m.
\]
We then check that this bound is sufficient unless $q = 3$ and $n \in \{10,14,18\}$. For $(n,q)=(10,3)$ we can appeal to Proposition \ref{p:comp}. And in the two remaining cases, by setting $a_3 = 2^{3m}(m!)_2$, we can check that the bound in \eqref{e:three} is good enough. The result follows.
\end{proof}

\vs

This completes the proof of Theorem \ref{t:main2}. By combining this result with Theorem \ref{t:zenkov}, we conclude that the proof of Theorem \ref{t:main1} is complete.

\end{document}